\documentclass[11pt]{article}
\usepackage[utf8]{inputenc}
\usepackage[T1]{fontenc}
\usepackage{graphicx,subcaption}
\usepackage{longtable,booktabs}
\usepackage{wrapfig}
\usepackage{rotating}
\usepackage[normalem]{ulem}
\usepackage{amsmath,amsthm}
\usepackage{amssymb}
\usepackage{capt-of}
\usepackage[margin=0.8in]{geometry}
\usepackage{pdflscape}
\usepackage{amsmath}
\usepackage{hyperref}
\usepackage{enumitem}
\usepackage{tabularx}
\usepackage{cleveref}

\usepackage{mathtools,mathrsfs,bm}
\usepackage{graph-separation}

\usepackage{tikz}
\usetikzlibrary{cd,decorations.pathmorphing}

\usepackage[
backend=bibtex,
style=authoryear-comp,
citestyle=authoryear-comp,
url=false,
uniquename=false
]{biblatex}

\setlist[enumerate]{noitemsep}
\setlist[description]{noitemsep}
\setlist[itemize]{noitemsep}
\theoremstyle{plain}
\newtheorem{theorem}{Theorem}
\newtheorem{proposition}{Proposition}

\theoremstyle{remark}
\newtheorem{remark}{Remark}

\DeclareMathAlphabet{\mathpzc}{OT1}{pzc}{m}{it}
\newcommand{\id}{\textnormal{id}} 
\newcommand{\Id}{\textnormal{Id}} 
\newcommand{\GLS}{\mathbb{N}}

\DeclareMathOperator{\normal}{\mathrm{N}} 
\DeclareMathOperator{\diag}{diag}

\newcommand{\tildegG}{\,\tilde{\gG}}
\newcommand{\ingraph}[1]{~\textnormal{\bf in}\,#1}

\newcommand{\textnot}{\textnormal{\bf not}~}
\newcommand{\underdist}[1]{~\textnormal{\bf under}\,#1}

\newcommand{\DMG}{\mathbb{G}}
\newcommand{\DMGc}{\mathbb{G}^{\ast}}
\newcommand{\DiG}{\mathbb{G}^{\ast}_{\textnormal{D}}}
\newcommand{\ADMG}{\mathbb{G}_{\textnormal{A}}}
\newcommand{\ADMGc}{\mathbb{G}^{\ast}_{\textnormal{A}}}
\newcommand{\DAG}{\mathbb{G}^{\ast}_{\textnormal{DA}}}
\newcommand{\UG}{\mathbb{UG}}
\newcommand{\BDG}{\mathbb{G}^{\ast}_{\textnormal{B}}}

\let\P\relax
\DeclareMathOperator{\P}{\mathsf{P}} 

\DeclareMathOperator{\marg}{margin}
\DeclareMathOperator{\trim}{trim}
\DeclareMathOperator{\augg}{augment}

\DeclareMathOperator{\gG}{\mathscr{G}}

\DeclareMathOperator{\sE}{\mathcal{E}}
\DeclareMathOperator{\sD}{\mathcal{D}}
\DeclareMathOperator{\sB}{\mathcal{B}}

\DeclareMathOperator{\sJ}{\mathcal{J}}

\DeclareMathOperator{\sK}{\mathcal{K}}

\DeclareMathOperator{\sL}{\mathcal{L}}

\DeclareMathOperator{\sW}{\mathcal{W}}
\DeclareMathOperator{\an}{\mathrm{an}}

\DeclareMathOperator{\Cov}{\mathsf{Cov}}

\DeclareMathOperator{\Var}{\mathsf{Var}}

\newcommand\independent{\protect\mathpalette{\protect\independenT}{\perp}}
\def\independenT#1#2{\mathrel{\rlap{$#1#2$}\mkern2mu{#1#2}}}
\makeatletter
\providecommand{\leftsquigarrow}{%
  \mathrel{\mathpalette\reflect@squig\relax}%
}
\newcommand{\reflect@squig}[2]{%
  \reflectbox{$\m@th#1\rightsquigarrow$}%
}
\makeatother
\author{Qingyuan Zhao\thanks{Statistical Laboratory, University of Cambridge,
    qyzhao@statslab.cam.ac.uk.}}
\date{\today}
\title{A matrix algebra for graphical statistical models}
\begin{document}

\maketitle
\begin{abstract}
  Directed mixed graphs permit directed and bidirected edges between
  any two vertices. They were first considered in the path analysis
  developed by Sewall Wright and play an essential role in statistical
  modeling. We introduce a matrix algebra for walks on such
  graphs. Each element of the algebra is a matrix whose entries are
  sets of walks on the graph from the corresponding row to the
  corresponding column. The matrix algebra is then generated by applying
  addition (set union), multiplication (concatenation), and transpose
  to the two basic matrices consisting of directed and bidirected
  edges. We use it to formalize, in the context of Gaussian linear
  systems, the correspondence between important graphical concepts
  such as latent projection and graph separation with important
  probabilistic concepts such as marginalization and (conditional)
  independence. In two further examples regarding confounder
  adjustment and the augmentation criterion, we illustrate how the
  algebra allows us to visualize complex graphical proofs. A
  ``dictionary'' and \LaTeX macros for the matrix algebra are provided
  in the Appendix.
\end{abstract}

\tableofcontents

\section{Introduction}
\label{sec:introduction}

The idea of using graphs to describe statistical relationships between
random variables can be traced back to the method of path analysis by
\textcite{wrightRelativeImportanceHeredity1920,
  wright34_method_path_coeff}. Although this approach was originally
proposed for applications in genetics, it became very influential over
the next half century in economics
\parencite{haavelmoStatisticalImplicationsSystem1943} and the social sciences
\parencite{duncanIntroductionStructuralEquation1975,joreskogAdvancesFactorAnalysis1979,bollenStructuralEquationsLatent1989}
under the name ``structural/simultaneous equation models''. Typically,
an acyclic directed graph is used to represent the causal relationship
between random variables, and each variable is modeled as a linear
combination of its parents with some additive Gaussian noise.
Alternatively, directed graphs can be used as the basis of a
``Bayesian network'' to represent how probability distributions factorize
\parencite{pearlBayesianNetwcrksModel1985,pearlProbabilisticReasoningIntelligent1988,neapolitanProbabilisticReasoningExpert1989}. This
can be further extended to represent causality via the causal Markov assumption
\parencite{spirtesCausationPredictionSearch1993,pearlCausalDiagramsEmpirical1995}
or the nonparametric structural equation model
\parencite{pearlCausalityModelsReasoning2000}. The last model is
intimately related to the ``potential outcomes'' or counterfactual
approach for causal inference in statistics that can be traced back to
\textcite{neyman1923application} and
\textcite{rubin1974estimating}. 

In the graphical models literature, a folklore is that many results for
nonparametric graphical models have their origins in Gaussian linear
systems. Often, results in Gaussian linear systems can be extended to
the nonparametric model with minor modifications. Some examples include global
identifiability
\parencite{drtonGlobalIdentifiabilityLinear2011,shpitserIdentificationGraphicalCausal2018},
nested Markov properties
\parencite{shpitserAcyclicLinearSEMs2018,richardsonNestedMarkovProperties2023},
proximal causal inference
\parencite{kurokiMeasurementBiasEffect2014,miaoIdentifyingCausalEffects2018},
sufficient adjustment sets
\parencite{perkovicCompleteGraphicalCharacterization2018}, and
efficient adjustment sets
\parencite{henckelGraphicalCriteriaEfficient2022,guoVariableEliminationGraph2023}. At
the very least, Gaussian linear systems provide concrete examples for more general
results about the nonparametric model, which is why
\textcite{pearlLinearModelsUseful2013} called linear models ``a useful
'microscope' for causal analysis''.

The purpose of this article is to develop a more formal connection
between Gaussian linear systems and nonparametric graphical
models. The main contribution is a matrix algebra, directly
motivated by Gaussian linear systems, that facilitates reasoning with
general graphical models. This algebra can help its users visualize
complex graphical concepts, build up intuitions, identify gaps in
existing proofs, and potentially derive new results.

The rest of this article is organized as
follows. \Cref{sec:gauss-line-syst} defines Gaussian linear systems
and reviews their probabilistic
properties. \Cref{sec:matr-algebra-direct} introduces the walk algebra
and \Cref{sec:import-types-walks} uses it to define some important
graphical concepts, which is then corresponded to probabilistic
concepts in \Cref{sec:conn-gauss-line} in the context of Gaussian
linear systems. \Cref{sec:further-examples} gives two more
``advanced'' examples to demonstrate how the matrix algebra
facilitates complex graphical reasoning. \Cref{sec:discussion}
concludes the article with some further discussion. As the main
purpose is to introduce and demonstrate the richness of the matrix
algebra, we will only sketch a proof for some of the technical results
below.

\section{Gaussian linear systems}
\label{sec:gauss-line-syst}

We say a real-valued $d$-dimensional random vector $V = (V_1, \dots, V_d)$ follows
a \emph{linear system} with parameters $\beta \in \mathbb{R}^{d
  \times d}$ and $\Lambda \in \mathbb{R}^{d \times d}$, $\Lambda$ is positive
semi-definite, if $V$ solves
\begin{equation}
  \label{eq:linear-system-matrix}
  V = \beta^T V + E,
\end{equation}
where $E$ is a $d$-dimensional random vector with mean $0$ and
covariance matrix $\Lambda$. We say the linear system is (principally)
\emph{non-singular} if (every principal sub-matrix of) $\Id - \beta$ is
invertible, and (principally) \emph{stable} if (every principal
sub-matrix of) $\beta$ has spectral radius less than $1$. We say the
linear system is \emph{Gaussian} if the distribution of $E$ is
multivariate normal.

When the linear system is non-singular and Gaussian, we immediately
have
\begin{equation}
  \label{eq:gls-dist}
  V = (\Id - \beta)^{-T} E \sim \normal(0,
  \Sigma),~\text{where}~\Sigma = (\Id - \beta)^{-T} \Lambda (\Id -
  \beta)^{-1},
\end{equation}
because the linear transformation of a normal random vector is still
normal. Here, $\Id$ is the identity matrix whose $(j,k)$-entry is
$\delta_{kW} = 1$ if $j = k$ and $0$ if $j \neq k$.
Thus, the probability distribution of $V$ is fully
characterized by its covariance matrix. The Neumann series $(\Id -
\beta)^{-1} = \sum_{q=0}^{\infty} \beta^q$, which holds when the
system is stable, allows us to connect Gaussian
linear systems to directed graphs. In particular, the covariance
matrix of $V$ has a graphical interpretation that shall be described
in detail \Cref{sec:conn-gauss-line}.

Marginalization is a basic operation in probability and statistics
that amounts to obtaining the ``marginal distribution'' of a subset of
random variables. Marginalization is particularly simple in Gaussian
models:
\[
  \text{if}~V \sim \normal(\mu, \Sigma),~\text{then}~V_{\sJ} \sim
  \normal(\mu_{\sJ}, \Sigma_{\sJ, \sJ}),
\]
where $\mu_{\sJ}$ and $\Sigma_{\sJ, \sJ}$ are the sub-vector/matrix
corresponding to the index set $\sJ \subseteq \{1,\dots,d\}$. This
will directly motivate a definition of marginalization for directed
graphs.

(Conditional) Independence is one of the key concepts that distinguish
probability from a general measure. Marginal and conditional
independence of random variables will be formally defined in the next
Section. In Gaussian models, marginal
independence can be read off from the covariance matrix: if $V \sim
\normal(\mu,\Sigma)$, then
\[
  V_{\sJ} \independent V_{\sK} \Longleftrightarrow \Sigma_{\sJ, \sK} = 0
\]
for all $\sJ, \sK \subseteq [d]$, where $\Leftrightarrow$ means
``if and only if''. It is well known that
pairwise conditional independence given all remaining variables can be
read off from the inverse covariance matrix: if $V \sim
\normal(\mu,\Sigma)$ and $\Sigma$ is positive definite, then
\[
  V_j \independent V_k \mid V_{[d] \setminus \{j,k\}}
  \Longleftrightarrow (\Sigma^{-1})_{jk} = 0
\]
for all $j,k \in [d]$, $j \neq k$. When a smaller subset of variables
$V_{\sL}$, $\sL \subseteq [d] \setminus \{j,k\}$ is conditioned on,
one can check the conditional independence $V_j \independent V_k \mid
V_{\sL}$ by first taking the principal sub-matrix of $\Sigma$ for $\sL
\cup \{j,k\}$ (corresponding to marginalization) and then checking the
$(j,k)$-entry of its inverse.

\section{The walk algebra on directed mixed graphs}
\label{sec:matr-algebra-direct}

Every Gaussian linear system defined in the previous section can be
associated with a mixed graph $\gG = (V, \sD, \sB)$; here the
denotation of the vertex set as $V$ is an intentional abuse of
notation. Specifically, the edge sets of $\gG$ is defined as
\begin{align*}
  (V_j,V_k) \in \sD~\text{(written as $V_j \rdedge V_k$)}
  \Longleftrightarrow \beta_{jk} \neq 0,\\
  (V_j,V_k) \in \sB~\text{(written as $V_j \bdedge V_k$)}
  \Longleftrightarrow \Lambda_{jk} \neq 0,
\end{align*}
for all $V_j,V_k \in V$. Because the covariance matrix $\Lambda$ of
$E$ is symmetric, edges in $\sB$ are symmetric. The choice of drawing
edges in $\sB$ as bidirected instead of undirected is intentional and
crucial, which is also why we shall call such $\gG$ is a
\emph{directed mixed graph} following
\textcite{richardson03_markov_proper_acycl_direc_mixed_graph}. We shall
see below that bidirected edges can be interpreted as latent variables
in a formal way. Note that loops are allowed in $\gG$. In fact,
bidirected loops (like $V_j \bdedge V_j$, meaning $V_j$ is influenced
by some exogenous noise) play an essential role in the development
below. There can be at most three edges between each pair of distinct
vertices: $V_j \rdedge V_k$, $V_j \ldedge V_k$, and $V_j \bdedge V_k$.

The matrix algebra introduced in this graph is generated by three basic
operations on two basic matrices. Each matrix in this algebra is $d
\times d$ with $(j,k)$-entry (or more accurately the
$(V_j,V_k)$-entry) being a set of walks from $V_j$ to $V_k$ for all
$j,k \in [d]$; a walk is an ordered sequence of connected
edges\footnote{Some authors use the term ``semi-walk'' to emphasize that edge
direction is ignored in determining connection.} and is best
defined using the concatenation operation introduced below. The two
basic matrices, $W[V\rdedge V \ingraph{\gG}]$ and $W[V\bdedge V
\ingraph{\gG}]$, collect the directed and bidirected edges
respectively (we will often omit the graph $\gG$ in the matrices if it
is clear from the context). Specifically, the $(j,k)$-entry of
$W[V\rdedge V]$ is given by
\begin{equation}
  \label{eq:directed-edges}
  W[V_j \rdedge V_k] =
  \begin{cases}
    \{V_j \rdedge V_k\},& \text{if}~(V_j,V_k) \in \sD, \\
    \emptyset,& \text{otherwise}, \\
  \end{cases}
\end{equation}
and the $(j,k)$-entry of $W[V\bdedge V]$ is given by
\begin{equation}
  \label{eq:bidirected-edges}
  W[V_j \bdedge V_k] =
  \begin{cases}
    \{V_j \bdedge V_k\},& \text{if}~(V_j,V_k) \in \sB, \\
    \emptyset,& \text{otherwise}. \\
  \end{cases}
\end{equation}
The three main binary operations on (sets of) walks are:
\begin{enumerate}
\item Set union (denoted as $+$): for example,
  \[
    \{V_2 \rdedge V_5\} + \{V_2 \rdedge V_3 \rdedge V_5\} = \{V_2
    \rdedge V_5,~V_2 \rdedge V_3 \rdedge V_5\}.
  \]
\item Concatenation of all pairs of walks from the two sets (denoted
  as $\cdot$): for example,
  \[
    \{V_2 \bdedge V_2\} \cdot \{V_2 \rdedge V_5, ~V_2 \rdedge V_3
    \rdedge V_5\} =
    \{V_2 \bdedge V_2 \rdedge V_5, ~V_2 \bdedge V_2 \rdedge V_3
    \rdedge V_5\}.
  \]
\item Transpose of a walk $w$, denoted as $w^T$, is obtained by
  writing it from right to left: for example,
\[
  \{V_2 \rdedge V_5, ~V_2 \rdedge V_3
  \rdedge V_5\}^T = \{V_5 \ldedge V_2,~V_5 \ldedge V_3 \ldedge V_2\}.
\]
\end{enumerate}
Using these, we can extend the definition of matrix addition,
multiplication, and transpose in linear algebra to matrices of sets of
walks:
\begin{align*}
  (W + W')[V_j,V_k] &= W[V_j,V_k] + W'[V_j,V_k], 
  \\
  (W \cdot W')[V_j,V_k] &= \sum_{V_l \in V} W[V_j,V_l] \cdot W'[V_l,V_k] 
                          , 
  \\
  (W^T)[V_j,V_k] &= (W[V_k,V_j])^T = \{w^T: w \in
  W[V_k,V_j]\}, 
\end{align*}
The \emph{walk algebra} over $\gG$ refers to
the collection of matrices that can be obtained by applying the above operations to $W[V
\rdedge V]$ and $W[V \bdedge V]$. It is easy to see that the
$(j,k)$-entry of every such matrix must be a set of walks from vertex $V_j$
to vertex $V_k$ or empty. Matrix subsetting (sub-matrices) can also be
defined in the natural way.

It is useful to also introduce, with an abuse of notation, the
identity matrix $\Id = \diag(\id, \dots, \id)$ for matrix
multiplication, where $\id$ is the trivial walk with length $0$ such
that
\begin{equation*}
  \id \cdot w = w \cdot \id = w\quad\text{and}\quad \Id \cdot W =
  W \cdot \Id = W
\end{equation*}
hold for any walk $w$ and matrix $W$ of sets of walks. When every
entry of $W$ is empty, we write $W = \emptyset$. This is the identity
element for matrix addition (set union).

\begin{remark} \label{rem:dioid}
The matrix algebra defined above is a ``dioid'' in the
terminology of \textcite{gondranGraphsDioidsSemirings2008} which means
the following axiomatic properties are satisfied:
\begin{enumerate}
\item $+$ is a commutative monoid (associative with identity element
  $\emptyset$);
\item $\cdot$ is a monoid (associative with identity element $\Id$);
\item $\cdot$ is distributive with respect to $+$;
\item The pre-order defined by $+$ ($W \preceq W'$ if and only if $W'
  = W + W''$ for some $W''$) is anti-symmetric: $W \preceq W'$ and $W'
  \preceq W$ imply that $W = W'$.
\end{enumerate}
\textcite[p.\ 28]{gondranGraphsDioidsSemirings2008} also requires the
identity element of addition to be absorbing for multiplication; this
can be met by introducing the convention that $W \cdot \emptyset =
\emptyset \cdot W = \emptyset$ for all matrices $W$. The fundamental
difference between rings and dioids (which are both semi-rings) is
that addition $+$ is a commutative group in the former but an ordered
monoid in the latter. A classical example of the former is the ring of
real square matrices, and another example of the latter is the dioid
of nonnegative square matrices.
Semirings and dioids are extensively studied in automata and formal
language theory \parencite{lothaireAppliedCombinatoricsWords2005} and
path finding problems in graphs
\parencite{gondranGraphsDioidsSemirings2008}, so the definition of
walk algebra above cannot be new. What is likely new to researchers
familiar with those problems is the consideration of directed mixed
graphs and special graphical structures motivated by key probabilistic
concepts such as marginalization and independence.
\end{remark}

Let us introduce some important subclasses of graphs. Let
$\mathbb{G}(V)$ denote all directed mixed graphs with vertex set
$V$. The following subclasses of $\mathbb{G}(V)$ will be useful later:
\begin{itemize}
\item $\DMGc(V)$: the class of canonical graphs with vertex set
  $V$ (by \emph{canonical}, we mean the graph contains all bidirected loops,
  that is, $(V_j,V_j) \in \sB$ for all $V_j \in V$);
\item $\BDG(V)$: the class of bidirected canonical graphs with
  vertex set $V$ (i.e.\ the directed edge set $\sD = \emptyset$);
\item $\DiG(V)$: the class of canonically directed graphs with vertex
  set $V$ (by \emph{canonically directed}, we mean the graph is
  canonical and contains no bidirected edges other than bidirected
  loops);
\item $\ADMG(V)$: the class of acyclic graphs with vertex set $V$ (by
  \emph{acyclic}, we mean there exists no cyclic directed walk like
  $V_j \rdedge \dots \rdedge V_j$ for some $V_j \in V$);
\item $\ADMGc(V) = \DMGc(V) \cap \ADMG(V)$: the class of canonical
  acyclic graphs with vertex set $V$;
\item $\DAG(V) = \DiG(V) \cap \ADMG(V)$: the class of canonically
  directed and acyclic graphs with vertex set $V$.
\end{itemize}

By trimming the bidirected loops, $\DiG(V)$ and $\DAG(V)$ are
isomorphic to the class of directed graphs and directed acyclic graphs
(DAG), respectively. The central object studied in the statistical
theory for causality \parencite[see
e.g.][]{pearlCausalityModelsReasoning2000,richardsonNestedMarkovProperties2023}
is the class of acyclic directed mixed graphs (ADMGs) that can be
obtained by trimming $\ADMGc(V)$.

\begin{figure}[t] \centering
  \begin{subfigure}[t]{0.45\textwidth} \centering
    \[
    \begin{tikzcd}
      V_1 \arrow[rd, blue] \arrow[rr, red, leftrightarrow, bend left
      = 30] \arrow[rrrd, red, leftrightarrow, bend left
      = 60, end anchor=north] \arrow[loop left, leftrightarrow, red] & &
      V_4 
      \arrow[loop left, leftrightarrow, red] &
      \\
      & V_3 \arrow[ru, blue] \arrow[rr, blue] \arrow[loop left,
      leftrightarrow, red] & & V_5 \arrow[loop right, leftrightarrow, red] \\
      V_2 \arrow[loop left, leftrightarrow, red] \arrow[ru, blue]
      \arrow[rrru, blue]  & & &
    \end{tikzcd}
    \]
    \caption{Canonical graph.
    }
    \label{fig:admg-example-1-full}
  \end{subfigure}
  \quad
  \begin{subfigure}[t]{0.45\textwidth} \centering
    \[
    \begin{tikzcd}
      V_1 \arrow[rd, blue] \arrow[rr, red, leftrightarrow, bend left
      = 30] \arrow[rrrd, red, leftrightarrow, bend left
      = 60, end anchor=north] & & V_4 
      &
      \\
      & V_3 \arrow[ru, blue] \arrow[rr, blue] & & V_5 \\
      V_2 \arrow[ru, blue] \arrow[rrru, blue] & & &
    \end{tikzcd}
    \]
    \caption{Trimmed graph.}
    \label{fig:admg-example-1-simple}
  \end{subfigure}

  \caption{Running example.}
  \label{fig:admg-example-1}
\end{figure}
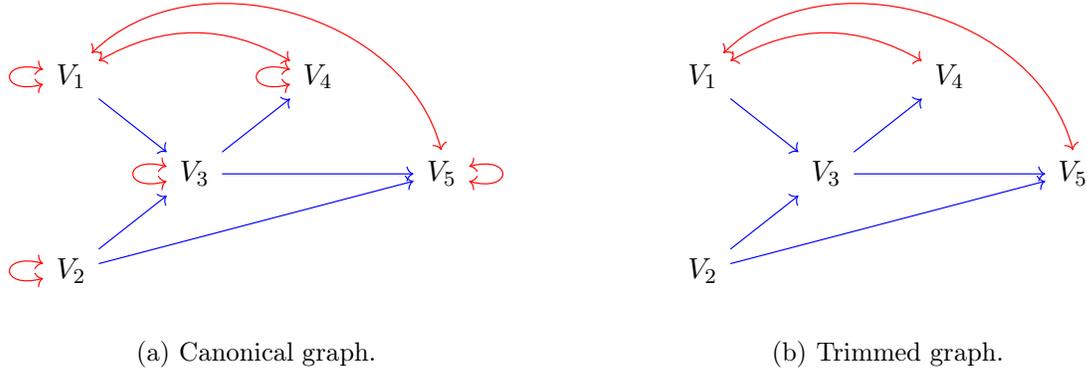

As an example, \Cref{fig:admg-example-1-simple} is obtained by
trimming the canonical graph in \Cref{fig:admg-example-1-full}, which
has the following directed edge matrix
  \[
    W[V \rdedge V] =
    \begin{pmatrix}
      \emptyset & \emptyset & \{V_1 \rdedge V_3\} & \emptyset &
      \emptyset \\
      \emptyset & \emptyset & \{V_2 \rdedge V_3\} & \emptyset &
      \{V_2 \rdedge V_5\} \\
      \emptyset & \emptyset & \emptyset & \{V_3 \rdedge V_4\} &
      \{V_2 \rdedge V_5\} \\
      \emptyset & \emptyset & \emptyset & \emptyset &
      \emptyset \\
      \emptyset & \emptyset & \emptyset & \emptyset & \emptyset
    \end{pmatrix}
  \]
and bidirected edge matrix
  \[
    W[V \bdedge V] =
    \begin{pmatrix}
      \{V_1 \bdedge V_1\} & \emptyset & \emptyset & \{V_1 \bdedge V_4\} &
      \{V_1 \bdedge V_5\} \\
      \emptyset & \{V_2 \bdedge V_2\} & \emptyset & \emptyset &
      \emptyset \\
      \emptyset & \emptyset & \{V_3 \rdedge V_3\} & \emptyset &
      \emptyset \\
      \{V_4 \bdedge V_1\} & \emptyset & \emptyset & \{V_4
      \bdedge V_4\} &
      \emptyset \\
      \{V_5 \bdedge V_1\} & \emptyset & \emptyset & \{V_5
      \bdedge V_4\} & \{V_5 \bdedge V_5\}
    \end{pmatrix}.
  \]
We will use this as a running example below.

\section{Important graphical concepts}
\label{sec:import-types-walks}

Many important concepts in graphical statistical models are about the
lack of certain type of walks in the graph between subsets of
vertices. Defining these graphical concepts in words is often awkward
and imprecise, but the matrix algebra introduced above allows us to
clearly define and visualize the graphical concepts.

\subsection{Special types of arcs}
\label{sec:arcs}

Given $\gG \in \DMG(V)$, \emph{(right-)directed walks} are defined as
elements of the matrix
\begin{equation}
  \label{eq:directed-walk}
  W[V \rdpath V] = \sum_{q = 1}^{\infty} (W[V
  \rdedge V])^q,
\end{equation}
where $(W[V \rdedge V])^q$ is obtained by
multiplying the directed edge matrix $q$ times and contains all
directed walks of length $q$. Thus, the graph $\gG$ is acyclic if and
only if the diagonal entries of $W[V \rdpath V]$ are
empty. \emph{Left-directed walks} can be defined as
elements of the entries of the transpose of \eqref{eq:directed-walk}:
\begin{equation}
  \label{eq:left-directed-walk}
  W[V \ldpath V] = (W[V \rdpath V])^T.
\end{equation}

\emph{Treks} or \emph{t-connected walks} are defined as elements of
\begin{align}
  W[V \trek V]
  =& (\Id + W[V \ldpath V]) \cdot W[V
     \bdedge V] \cdot (\Id + W[V
     \rdpath V]) \label{eq:trek} \\
  =& W[V \bdedge V] + W[V \bdedge V
     \rdpath V] + W[V \ldpath V \bdedge
     V] + W[V \ldpath V \bdedge V \rdpath
    V], \tag*{}
\end{align}
where $W[V \bdedge V \rdpath V]$ is a shorthand notation for $W[V
\bdedge V] \cdot W[V \rdpath V]$ and likewise for the other
matrices. This definition is directly motivated by the covariance
matrix of Gaussian linear systems in \eqref{eq:gls-dist} and this
connection will be explored in the next Section. We say a walk is
\emph{d-connected} if it can be obtained
from removing the bidirected edge in a trek. It follows from
\eqref{eq:trek} that
the matrix of d-connected walks can be written as
\begin{align}
  W[V \dconnarc V] =& (\Id + W[V \ldpath V]) \cdot (\Id +
                      W[V \rdpath V]) \setminus \Id \tag*{} \\
  =& W[V \ldpath V] + W[V \rdpath V] + W[V
     \ldpath V \rdpath V], \label{eq:dconn-walk}
\end{align}
where set difference $\setminus$ with the identity matrix $\Id$ in the
first equation is applied entry-wise. Half arrowheads are used in $W[V
\dconnarc V]$ to emphasize that these walks may or may not have an
arrowhead at the two ends. We say a walk is an \emph{arc} or
\emph{m-connected} if it is a trek or d-connect, so it belongs to
\begin{equation}
  \label{eq:mconn-walk}
  W[V \mconnarc V] = W[V \trek V] + W[V \dconnarc V].
\end{equation}

A key notion in graphical statistical models is collider. We say a
vertex $V_j$ is a \emph{collider} in a walk if the walk contains two
arrowheads ``colliding'' at $V_j$ like
\[
  \rdedge V_j \ldedge,~\bdedge V_j
  \ldedge,~\text{or}~\bdedge V_j \bdedge.
\]
We use a squiggly line to $\nosquigno$ to denote a walk with no
collider. The next result justifies this and further characterizes
t/d/m-connections.

\begin{proposition} \label{prop:arc-collider}
  The following statements are true in any directed mixed graph.
  \begin{enumerate}
  \item A walk is t-connected (aka a trek) if and only if it has no
    collider and exactly one bidirected edge.
  \item A walk is d-connected if and only if it contains no
    collider and no bidirected edge.
  \item A walk is m-connected (aka an arc) if and only if it contains
    no collider.
  \end{enumerate}
\end{proposition}

Conditioning is a key concept in probability theory, and there is an
intimately related concept in causal graphical models. We say an arc
in $\gG \in \DMG(V)$ is \emph{blocked} by $L \subseteq V$ if the arc
has an non-endpoint that is in $L$.  So the walks $V_1
\rdedge V_3 \rdedge V_4$ and $V_4 \ldedge V_3 \bdedge V_3
\rdedge V_5$ are blocked by $V_3$. Heuristically, conditioning on a
non-endpoint stops information from flowing along that arc, hence the
name blocking.

We introduce some additional notation related to blocking. Let
$W[V \rdpath V \mid L \ingraph{\gG}]$ denote the matrix whose
$(j,k)$-entry is the set of right directed walks from $V_j$ to $V_k$ in
$\gG$ that are not blocked by $L$. Using our matrix notation, this can
be written as
\begin{equation}
  \label{eq:unblocked-directed-walks}
  W[V \rdpath V \mid L] = W[V \rdedge V] +
  W[V \rdedge L^c] \cdot (\Id + W[L^c \rdpath L^c \mid L]) \cdot
  W[L^c \rdedge V],
\end{equation}
where $L^c = V \setminus L$ and
\begin{equation}
  \label{eq:unblocked-directed-walk}
  W[L^c \rdpath L^c \mid L]) = \sum_{q=1}^{\infty} (W[L^c
  \rdedge L^c])^q.
\end{equation}
This notation naturally extends to other types of arcs. For example,
$W[V \trek V \mid L \ingraph{\gG}]$ is the matrix of unblocked treks
and can be formally defined as
\begin{equation}
  \label{eq:unblocked-treks}
  \begin{split}
    W[V \trek V \mid L]
    =& W[V \bdedge V] + W[V
    \bdedge L^c] \cdot W[L^c \rdpath V \mid L] +
    W[V \ldpath L^c] \cdot W[L^c \bdedge
    V \mid L] \\
    &+ W[V \ldpath L^c \mid L] \cdot W[L^c
    \bdedge L^c] \cdot W[L^c \rdpath
    V \mid L].
  \end{split}
\end{equation}

The next result relate these different notions of connections in
canonical graphs. Let $\trim$ denote the operation that removes all
bidirected loops in a graph. Let $\mathcal{P}_{\gG}$ denote the set of
all paths in $\gG$, and let ``$P$ matrices'' be defined as the intersection of $\mathcal{P}_{\gG}$
with the corresponding ``$W$ matrices''; for example,
\begin{align}
  P[V \rdpath V] &= W[V \rdpath V] \cap \mathcal{P}_{\gG},\label{eq:directed-path}\\
  P[V \trek V] &= W[V \trek V] \cap
                 \mathcal{P}_{\gG}.\label{eq:trekpath-given-L}
\end{align}
where intersection is applied entry-wise. Obviously, $\gG$ is acyclic
if and only if $W[V \rdpath V] = P[V \rdpath V]$.

\begin{proposition} \label{prop:trek-exist}
  Consider $\gG \in \DMG(V)$ and any disjoint $\{j\}, \{k\}, L \subset
  V$. If $\gG \in \DMGc(V)$, then the following are equivalent:
  \begin{enumerate}
  \item $j \trek k \mid L \ingraph{\gG}$;
  \item $j \mconnarc k \mid L \ingraph{\gG}$;
  \item $P[j \mconnarc k \mid L  \ingraph{\gG}] \neq \emptyset$;
  \item $j \mconnarc k \mid L \ingraph{\trim(\gG)}$;
  \item $P[j \mconnarc k \mid L \ingraph{\trim(\gG)}] \neq \emptyset$.
  \end{enumerate}
  If further $\gG \in \DAG(V)$, these statements are also equivalent to
  \begin{enumerate}[resume]
  \item $j \dconnarc k \mid L \ingraph{\gG}$;
  \item $P[j \dconnarc k \mid L  \ingraph{\gG}] \neq \emptyset$;
  \item $j \dconnarc k \mid L \ingraph{\trim(\gG)}$;
  \item $P[j \dconnarc k \mid L \ingraph{\trim(\gG)}] \neq \emptyset$.
  \end{enumerate}
\end{proposition}

\subsection{Marginalization of directed mixed graphs}
\label{sec:marg-direct-mixed}

As mentioned above, bidirected edges in directed mixed graphs
represent latent variables intuitively. To formalize this intuition,
it is helpful to think about the existence of certain types of walks
as a relation. For example, define
\[
  J \rdedge K \ingraph{\gG} \Longleftrightarrow W[J
  \rdedge K \ingraph{\gG}] \neq \emptyset;
\]
that is, $J \rdedge K \ingraph{\gG}$ means that there exists an edge
$V_j \rdedge V_j$ for $V_j \in J$ and $V_k \in K$. Other relations can
be defined similarly; for example,
\begin{align*}
  J \rdpath K \mid L\ \text{in}\ \gG &\Longleftrightarrow \sW[J
                                       \rdpath K \mid L \ingraph{\gG}] \neq \emptyset, \\
  J \trek K \mid L \ \text{in}\ \gG &\Longleftrightarrow \sW[J
                                      \trek K \mid L \ingraph{\gG}]
                                      \neq \emptyset.
\end{align*}
We use the convention that none of the above connections holds when
$J$ or $K$ is empty.

\emph{Marginalization} (or \emph{latent projection}) of directed mixed
graphs was first introduced by
\textcite{vermaEquivalenceSynthesisCausal1990}. Formally, it is a
mapping between graphs
  \begin{align*}
    \marg_{\tilde{V}}: \DMG(V) &\to \DMG(\tilde{V}),    \\
    \gG &\mapsto \tilde{\gG},
  \end{align*}
where $\tilde{V} \subseteq V$ and the edges in $\tilde{\gG} =
\marg_{\tilde{V}}(\gG)$ are obtained by
\begin{align*}  
  j \rdedge k \ingraph{\tildegG}&\Longleftrightarrow j
                                  \rdpath k \mid \tilde{V}
                                  \ingraph{\gG},\quad j,k \in
                                  \tilde{V}, \\ 
  j
  \bdedge k \ingraph{ \tildegG} &\Longleftrightarrow j
                                  \trek k \mid \tilde{V}
                                  \ingraph{\gG},\quad j,k \in
                                  \tilde{V}. 
  \end{align*}
These equations can also also be viewed as a map from
walks in $\gG$ to walks in $\tilde{\gG}$ and written as
\begin{equation} \label{eq:latent-proj-walk}
  \marg_{\tilde{V}}:~W\Bigg[\tilde{V}
  \begin{Bmatrix}
    \rdpath \\
    \ldpath \\
    \trek \\
  \end{Bmatrix}
  \tilde{V} \mid \tilde{V} \ingraph{\gG}\Bigg] \mapsto W\Bigg[\tilde{V}
  \begin{Bmatrix}
    \rdedge \\
    \ldedge \\
    \bdedge \\
  \end{Bmatrix}
  \tilde{V} \ingraph{\tildegG}\Bigg],
\end{equation}
where $\marg_{\tilde{V}}$ is applied to each entry of each row in this
equation. This has a natural extension to those walks in $\gG$ that
can be written as the concatenation of one or multiple arcs (left hand
side of \eqref{eq:latent-proj-walk}) separated by colliders. 
The marginal walk is simply defined by concatenating the margin of each
component; for example,
\[
  \marg_{\tilde{V}}(W[\tilde{V} \rdpath \tilde{V} \mid \tilde{V} \ingraph{\gG}
  ] \cdot W[\tilde{V} \ldpath \tilde{V} \mid \tilde{V} \ingraph{\gG}
  ]) = W[\tilde{V} \rdedge \tilde{V} \ingraph{\tildegG}
  ] \cdot W[\tilde{V} \ldedge \tilde{V} \ingraph{\tildegG}
  ].
\]
Note that not all
walks in $\gG$ have a marginal walk (or we can say their margin walk
is the empty set). For
example, walks in $W[\tilde{V} \rdedge U \ldedge
\tilde{V} \ingraph{\gG}]$ where $U = V \setminus \tilde{V}$ cannot be
written as the concatenation of walks on the left hand side of
\eqref{eq:latent-proj-walk}. On the other hand, every walk in
$\tilde{\gG}$ must have at least one non-empty pre-image in $\gG$.

Marginalization preserves directed walks and treks in the following sense.


\begin{proposition}[Marginalization preserves unblocked directed walks
  and treks] \label{prop:marg-preserve-walks-1}
  For any $\gG \in \DMG(V)$ and $L \subseteq \tilde{V} \subseteq
  V$, we have
  \begin{align*}
    \marg_{\tilde{V}}\left(W \left[\tilde{V}
    \begin{Bmatrix}
      \rdpath \\
      \ldpath \\
      \trek
    \end{Bmatrix}
    \tilde{V} \mid L
    \ingraph{\gG} \right] \right) &= W \left[\tilde{V}
    \begin{Bmatrix}
      \rdpath \\
      \ldpath \\
      \trek
    \end{Bmatrix}
    \tilde{V} \mid L
    \ingraph{\marg_{\tilde{V}}(\gG)} \right],
  \end{align*}
  where the equality hold for each row of the equation.
  As a consequence, for any $J,K,L \subseteq \tilde{V} \subseteq V$ we
  have the following equivalences:
  \[
    J
    \begin{Bmatrix}
      \rdpath \\
      \ldpath \\
      \trek
    \end{Bmatrix}
    K \mid L \ingraph{\marg_{\tilde{V}}(\gG)}
    \Longleftrightarrow J     \begin{Bmatrix}
      \rdpath \\
      \ldpath \\
      \trek
    \end{Bmatrix} K \mid L \ingraph{\gG}.
  \]
\end{proposition}

It follows from setting $J = K = \{V_j\}$ and $L = \emptyset$ in
\Cref{prop:marg-preserve-walks-1} that marginalization preserves
acyclicity of the
graph:
\[\gG \in \ADMG(V) \Longrightarrow \marg_{\tilde{V}}(\gG) \in
  \ADMG(\tilde{V}),
\]
where $\Longrightarrow$ means ``implies''. Marginalization is
obviously idempotent: $\marg_{\tilde{V}} \circ
\marg_{\tilde{V}} = \marg_{\tilde{V}}$ for
all $\tilde{V} \subseteq V$, where $\circ$ means composition of
maps. Another corollary of Proposition
\ref{prop:marg-preserve-walks-1} is that the
order of marginalization does not matter in the sense that
\begin{equation*}
  \marg_{\tilde{V}'} = \marg_{\tilde{V}'} \circ
  \marg_{\tilde{V}},~\text{for all}~\tilde{V}' \subseteq \tilde{V}
  \subseteq V.
\end{equation*}
This shows that marginalization of directed mixed graphs, just like
marginalization of probability distributions, is like a projection.

Marginalization can be simplified in acyclic canonical graphs (aka
ADMGs). Let $\trim$ denote the operation that removes all bidirected
loops in a graph. For $\gG^{*} = \trim(\gG)$ where $\gG \in
\ADMGc(V)$, its marginal graph $\tilde{\gG}^{*} =
\marg^*_{\tilde{V}}(\gG^{*})$ is defined as the trimmed graph with
vertex set $\tilde{V} \subseteq V$ and edges
\begin{align*}  
  V_j \rdedge V_k \ingraph{\tildegG^{*}}&\Longleftrightarrow
                           P[V_j \rdpath V_k \mid \tilde{V}
    \ingraph{\gG^{*}}] \neq \emptyset, \\ 
  V_j \bdedge V_k \ingraph{\tildegG^{*}}&\Longleftrightarrow
                           P[V_j \confarc V_k \mid \tilde{V}
    \ingraph{\gG^{*}}] \neq
                                                     \emptyset. 
\end{align*}
This definition of marginalization for ADMGs is consistent with the
more general definition earlier because it can be shown that the
following commutative diagram holds:
\begin{center}
\begin{tikzcd}
  \ADMGc(V) \arrow[r, "\trim"] \arrow[d, swap,
  "\marg_{\tilde{V}}"] & \trim(\ADMGc(V)) \arrow[d,
  "\marg^{*}_{\tilde{V}}"] \\
  \ADMGc(\tilde{V}) \arrow[r, mapsto, "\trim"] &
  \trim(\ADMGc(\tilde{V}))
\end{tikzcd}
\quad
\begin{tikzcd}
  \gG \arrow[r, mapsto, "\trim"] \arrow[d, mapsto, swap,
  "\marg_{\tilde{V}}"] & \gG^{*} \arrow[d, mapsto,
  "\marg^{*}_{\tilde{V}}"] \\
  \tilde{\gG} \arrow[r, mapsto, "\trim"] & \tilde{\gG}^{*}
\end{tikzcd}
\end{center}

The next result is an extension of \Cref{prop:marg-preserve-walks-1}
to ADMGs.

\begin{proposition} \label{prop:marg-preserve-walks-3}
  For any $\gG^{*} \in \trim(\ADMGc(V))$ and $L \subseteq
  \tilde{V} \subseteq V$, we have
  \begin{align*}
    \marg^{*}_{\tilde{V}}\left(P \left[\tilde{V}
    \begin{Bmatrix}
      \rdpath \\
      \ldpath \\
      \confarc \\
      \mconnarc
    \end{Bmatrix}
    \tilde{V} \mid L
    \ingraph{\gG^{*}} \right] \right) &= P \left[\tilde{V}
    \begin{Bmatrix}
      \rdpath \\
      \ldpath \\
      \confarc \\
      \mconnarc
    \end{Bmatrix}
    \tilde{V} \mid L
    \ingraph{\marg^{*}_{\tilde{V}}(\gG)} \right]
  \end{align*}
  where the equality hold for each row of the equation.
\end{proposition}

\subsection{Graph separation}
\label{sec:graph-separation}

The last important concept we will introduce is graph separation. This
is based the following extension of the concept of blocking from arcs
to walks. We say a walk in $\gG \in \DMG(V)$ is \emph{blocked} by $L
\subseteq V$, if
\begin{enumerate}
\item the walk contains a collider $V_m$ such that $V_m \not \in L$, or
\item the walk contains a non-colliding non-endpoint $V_m$ such that $V_m
  \in L$ (any vertices not at the two ends are called a
  \emph{non-endpoint} of the walk).
\end{enumerate}
This definition is consistent with our earlier definition of blocking
an arc, which only requires the second condition because an arc has no
colliders.

For any $V_j,V_k \in V$, we denote the set of walks from $V_j$ to
$V_k$ in $\gG$ that are not blocked by $L$ as $W[V_j \mconn V_k \mid L
\ingraph{\gG}]$. If this set is empty, we say $V_j$ and $V_k$ are
\emph{m-separated} given $L$ in $\gG$ and write $\textnot V_j \mconn V_k \mid
L \ingraph{\gG}$. This definition naturally extends to sets of
variables using our matrix notation. For example, for $J,K \subseteq
V$ we write
\begin{align*}
  \textnot J \mconn K \mid L \ingraph{\gG} &\Longleftrightarrow W[J \mconn K \mid
                                   L \ingraph{\gG}] = \emptyset \\
                                 &\Longleftrightarrow \textnot V_j \mconn V_k \mid L
                                   \ingraph{\gG},~\text{for
                                   all}~V_j \in J, V_k \in K.
\end{align*}
As before, we often omit the graph $\gG$ when it can be inferred from
the context.

\begin{remark}
Any walk in $\gG$
can always be uniquely written as the concatenation of one or multiple
arcs, where any two consecutive arcs meet at a collider. 
In other words, any walk from $j$ to $k$ has the form $j \mconn k$,
where a wild card character $\ast$ is used in this notation to
emphasize that the walk may have $0$ or more colliders, and half
arrowheads are used to indicate that there may or may not exist
arrowheads at the two ends. So using our matrix notation, m-separation
can be formally defined by vanishing entries in
\begin{equation}
  \label{eq:mconn}
    W[V \mconn V \mid L]
    =W[V \mconnarc V \mid L] + W[V \halfsquigfull
    L \mid L] \cdot \Big\{ \Id + \sum_{q=1}^{\infty} (W[L \confarc L
    \mid L])^q \Big\} \cdot W[L \fullsquighalf V \mid L].
  \end{equation}
\end{remark}

\begin{remark}
We shall distinguish between
\[
  V_j \mconn V_k \quad \text{and} \quad V_j \mconn V_k \mid \emptyset,
\]
where the first means that there exists a walk from $V_j$ to $V_k$ in
$\gG$, and the second means that there exists such a walk that is not
blocked by the emptyset (so it must be an arc). In other words, $W[V
\mconn V]$ contains all walks in $\gG$ and $W[V \mconn V \mid
\emptyset] = W[V \mconnarc V]$.
\end{remark}

We introduce two useful variants of m-separation. Let $W[V_j \tconn V_k
\mid L]$ denote the set of \emph{t-connected} walks from $V_j$ to $V_k$
that is unblocked by $L$ and consists of one or several treks. In our
matrix notation, this is the $(j,k)$-entry of
\begin{equation}
  \label{eq:tconn}
    W[V \tconn V \mid L]
    =W[V \trek V \mid L] 
    + W[V \trek
    L \mid L] \cdot \Big\{ \Id + \sum_{q=1}^{\infty} (W[L \trek L
    \mid L])^q \Big\} \cdot W[L \trek V \mid L].
\end{equation}
If this set is empty in $\gG$, we say $V_j$ and $V_k$ are \emph{t-separated}
given $L$ and write $\textnot V_j \tconn V_k \mid L
\ingraph{\gG}$. Similarly, let $W[V_j \dconn V_k \mid L
\ingraph{\gG}]$ denote the set of walks from $V_j$
to $V_k$ consisting of directed edges only that is not blocked by
$L$. If this set is empty, we say $V_j$ and $V_k$ are \emph{d-separated}
in $\gG$ given $L$ and write $\textnot V_j \dconn V_k \mid L
\ingraph{\gG}$. Similar to t/m-separation, this definition naturally
extends to sets of vertices.

The next result follows from applying
\Cref{prop:marg-preserve-walks-1} to \eqref{eq:tconn}.

\begin{proposition}[Marginalization preserves t-connection] \label{prop:marg-preserve-walks-2}
  For any $\gG \in \DMG(V)$ and $J, K,L \subseteq \tilde{V} \subseteq
  V$, we have
  \[
    J \tconn K \mid L \ingraph{\marg_{\tilde{V}}(\gG)}
    \Longleftrightarrow J \tconn K \mid L \ingraph{\gG}.
  \]
\end{proposition}

An experienced reader will find the above definition of blocking
slightly different from many other authors who define graph separation
using paths (walks without duplicated vertices)
\parencite{lauritzenGraphicalModels1996,pearlCausalityModelsReasoning2000}. Specifically,
we say a walk in $\gG$ is \emph{ancestrally blocked} by $L$ if
\begin{enumerate}
\item the walk contains a collider $V_m$ such that $V_m \not \in L$ and $m
  \nordpath L \ingraph{\gG}$, or
\item the walk contains a non-colliding non-endpoint $V_m$ such that $m
  \in L$.
\end{enumerate}
It is obvious that an ancestrally blocked walk is also blocked in our
sense, but these two concepts are different. For example, in the
following graph
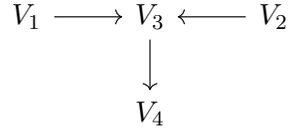
\[
  \begin{tikzcd}
    V_1 \arrow[r] & V_3 \arrow[d] & V_2 \arrow[l] \\
    & V_4 &
  \end{tikzcd}
\]
the path $V_1 \rdedge V_3 \ldedge V_2$ is blocked by $V_4$
but not ancestrally blocked by $V_4$. Nevertheless, in
our definition we still have $V_1 \mconn V_2 \mid V_4$ because the walk $V_1
\rdedge V_3 \rdedge V_4 \ldedge V_3 \ldedge V_2$ is not blocked by
$V_4$. Proposition \ref{prop:tsep-simplify} below shows that these two
notions of block define the same notion of m-separation for canonical
graphs. The advantage of our weaker notion of blocking is that it is
entirely a property of the walk and the set of vertices being
conditioned on. In contrast, ancestral blocking depends on the ambient
graph.

The ``$P$ matrices'' definition in \eqref{eq:directed-path}
and \eqref{eq:trekpath-given-L} naturally extend to walks with
colliders. For example,
\begin{align}
  P[V \mconn V \mid_a L] &= W[V \mconn V \mid_a L] \cap
\mathcal{P}_{\gG}, \label{eq:mconnpath-given-L} \\
  P[V \dconn V \mid_a L] &= W[V \dconn V \mid_a L] \cap
\mathcal{P}_{\gG}, \label{eq:dconnpath-given-L}
\end{align}
where intersection is applied entry-wise and $W[V \mconn V \mid_a L]$
is the matrix of set of walks that are not ancestrally blocked by
$L$.

In the literature, m-separation and d-separation are usually
defined, respectively, as the negation of (iii) and (v) in the next
Proposition.

\begin{proposition} \label{prop:tsep-simplify}
  Consider $\gG \in \DMG(V)$ and disjoint $\{V_j\}, \{V_k\}, L \subset
  V$. If $\gG$ is canonical, that is, $\gG \in \DMGc(V)$, we have
  \[
    \textnormal{(i)}~V_j \tconn V_k \mid L \Longleftrightarrow
    \textnormal{(ii)}~V_j \mconn V_k \mid L \Longleftrightarrow
    \textnormal{(iii)}~P[j \mconn k \mid_a L] \neq \emptyset.
  \]
  Furthermore, if $\gG$ is canonically directed, that is, $\gG \in
  \DiG(V)$, we have
  \[
    \textnormal{(i), (ii), (iii)} \Longleftrightarrow
    \textnormal{(iv)}~V_j \dconn V_k \mid L \Longleftrightarrow
    \textnormal{(v)}~P[V_j \dconn V_k \mid_a L] \neq \emptyset.
  \]
\end{proposition}

By \Cref{prop:marg-preserve-walks-2,prop:tsep-simplify}, m-connection
(d-connection) is also preserved by marginalization when the graph is
canonical (canonically directed).

\begin{proposition}[Marginalization preserves m- and
  d-connections] \label{cor:marg-preserve-m-connection}
  Consider $\gG \in \DMG(V)$ and $J, K,L \subseteq \tilde{V}
  \subseteq V$. If $\gG \in \DMGc(V)$, then
  \[
    J \mconn K \mid L \ingraph{\marg_{\tilde{V}}(\gG)}
    \Longleftrightarrow J \mconn K \mid L \ingraph{\gG}.
  \]
  Furthermore, if $\gG \in \DiG(V)$, then
  \[
    J \dconn K \mid L \ingraph{\marg_{\tilde{V}}(\gG)}
    \Longleftrightarrow J \dconn K \mid L \ingraph{\gG}.
  \]
\end{proposition}

When $J$ and $K$ are singletons, m-connection has a particularly
simple form in the smallest possible marginal graph. Let $W[V
\colliderconn V \ingraph{\gG}]$ denote all walks whose non-endpoints
are all colliders. In other words,
\begin{equation}
  \label{eq:collider-conn}
  W[V \colliderconn V]
  = (\Id + W[V \rdedge V]) \cdot \Big\{\Id + \sum_{q=1}^{\infty} (W[V
     \bdedge V])^q\Big\}
     \cdot (\Id + W[V \ldedge V]) \setminus \Id,
\end{equation}
where the set difference with the identity matrix is applied
entry-wise. This is closely related to the notion of ``Markov
blanket'' in \textcite{richardsonNestedMarkovProperties2023}. The next result
follows from \Cref{prop:marg-preserve-walks-1} and
\Cref{prop:marg-preserve-walks-2}.

\begin{proposition} \label{cor:mconn-colliderconn}
  Consider $\gG \in \DMGc(V)$ and any disjoint $\{V_j\},\{V_k\},L \subset
  V$. Let $\tilde{V} = \{V_j, V_k\} \cup L$, then
  \[
    V_j \mconn V_k \mid L \ingraph{\gG} \Longleftrightarrow V_j \colliderconn
    V_k \ingraph{\marg_{\tilde{V}}(\gG)}.
  \]
\end{proposition}

Finally, define
\begin{equation}
  \label{eq:samedist}
  W[V \samedist V] = \sum_{q=1}^{\infty} (W[V \bdedge V])^q.
\end{equation}
This is closely related to the notion of ``districts''
in \textcite{richardsonNestedMarkovProperties2023} or ``c-components''
(for confounded component) in
\textcite{tianTestableImplicationsCausal2002b}, which are maximal
sets of vertices connected by bidirected edges.

\subsection{Examples}
\label{sec:examples-1}

Consider the graph $\gG$ in \Cref{fig:admg-example-1-full}. The matrix
of directed walks in this graph is given by
  \[
    W[V \rdpath V ]
    =
    \begin{pmatrix}
      \emptyset & \emptyset & \{V_1 \rdedge V_3\} & \{V_1
      \rdedge V_3 \rdedge V_4\} & \{V_1
      \rdedge V_3 \rdedge V_5\} \\
      \emptyset & \emptyset & \{V_2 \rdedge V_3\} & \{V_2 \rdedge V_3
      \rdedge V_4\} &
      \{V_2 \rdedge V_5,~V_2 \rdedge V_3 \rdedge V_5\} \\
      \emptyset & \emptyset & \emptyset & \{V_3 \rdedge V_4\} &
      \{V_3 \rdedge V_5\} \\
      \emptyset & \emptyset & \emptyset & \emptyset &
      \{V_4 \rdedge V_5\} \\
      \emptyset & \emptyset & \emptyset & \emptyset & \emptyset
    \end{pmatrix}.
  \]

Some treks in this
graph (entries in $W[V \trek V \ingraph{\gG}]$) include
  \begin{align*}
    W[V_1 \trek V_4 \ingraph{\gG}] &= \{V_1 \bdedge V_4,~ V_1 \bdedge V_1
                   \rdedge V_3 \rdedge V_4 \}, \\
    W[V_3 \trek V_3 \ingraph{\gG}] &= \{V_3 \bdedge V_3,~ V_3 \ldedge V_1
                   \bdedge V_1 \rdedge V_3,~ V_3 \ldedge V_2
                   \bdedge V_2 \rdedge V_3\}, \\
    W[V_3 \trek V_4 \ingraph{\gG}] &= \{V_3 \bdedge V_3 \rdedge V_4,~ V_3 \ldedge V_1
                   \bdedge V_4, \\
                   &\quad ~~~ V_3 \ldedge V_1 \bdedge V_1 \rdedge V_3
                     \rdedge V_4,~ V_3 \ldedge V_2 \bdedge V_2 \rdedge V_3
                     \rdedge V_4\}.
  \end{align*}
  The first and third set of treks have different endpoints and can be
  trimmed to the following m-connected paths in
  \Cref{fig:admg-example-1-simple}:
  \begin{align*}
    P[V_1 \mconnarc V_4 \ingraph{\trim(\gG)}] &= \{V_1 \bdedge V_4,~ V_1 \rdedge V_3
    \rdedge V_4 \}, \\
    P[V_3 \mconnarc V_4 \ingraph{\trim(\gG)}] &= \{V_3 \rdedge V_4,~ V_3 \ldedge V_1
    \bdedge V_4\}.
  \end{align*}

When obtaining the marginal of $\gG$ on $\tilde{V} = \{V_1,V_2,V_4,V_5\}$,
four new edges are introduced: $V_1 \rdedge V_4$ (from $V_1 \rdedge V_3
\rdedge V_4$), $V_1 \rdedge V_5$ (from $V_1 \rdedge V_3 \rdedge V_5$), $V_2 \rdedge
V_4$ (from $V_2 \rdedge V_3 \rdedge V_4$), and $V_4 \bdedge V_5$ (from the trek $V_4
\ldedge V_3 \bdedge V_3 \rdedge V_5$ or the d-connected path $V_4 \ldedge V_3
\rdedge V_5$). The resulting graph is shown in
\Cref{fig:admg-example-proj-1245} (after trimming) and more examples
are given in \Cref{fig:admg-example-proj}.

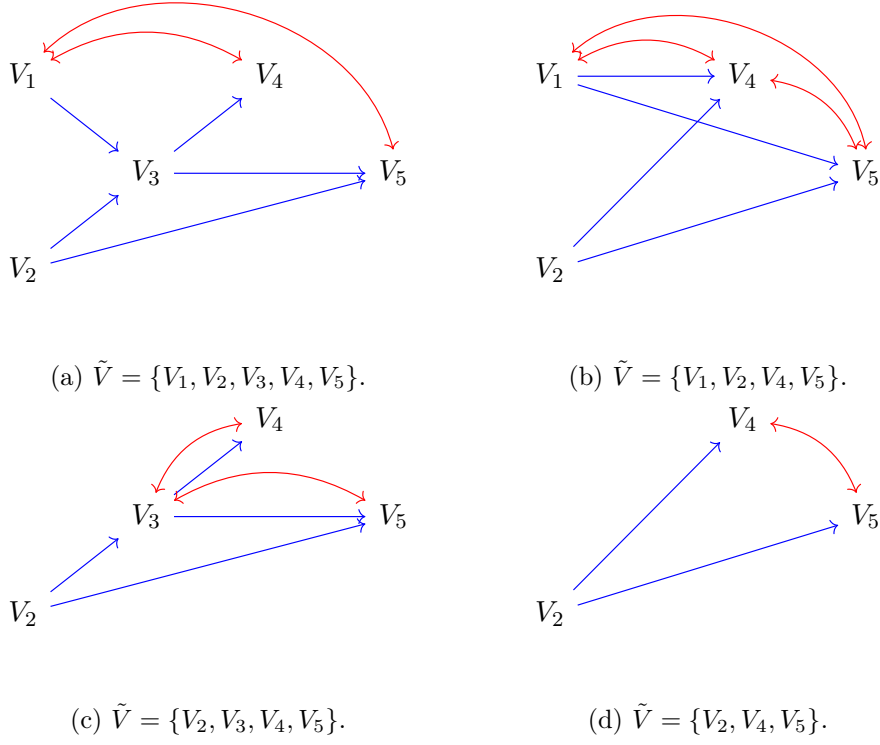
\begin{figure}[t] \centering
  \begin{subfigure}[b]{0.35\textwidth} \centering
    \begin{tikzcd}
      V_1 \arrow[rd, blue] \arrow[rr, red, leftrightarrow, bend left
      = 30] \arrow[rrrd, red, leftrightarrow, bend left
      = 60, end anchor=north] & & V_4 
      &
      \\
      & V_3 \arrow[ru, blue] \arrow[rr, blue] & & V_5 \\
      V_2 \arrow[ru, blue] \arrow[rrru, blue] & & & \\
    \end{tikzcd}
    \caption{$\tilde{V} = \{V_1,V_2,V_3,V_4,V_5\}$.}
    \label{fig:admg-example-proj-12345}
  \end{subfigure}
  \quad
  \begin{subfigure}[b]{0.35\textwidth} \centering
    \begin{tikzcd}
      V_1 \arrow[rr, blue] \arrow[rrrd, blue] \arrow[rr, red, leftrightarrow, bend left
      = 30] \arrow[rrrd, red, leftrightarrow, bend left
      = 60, end anchor=north]  & &
      V_4 \arrow[rd, red, leftrightarrow,
      bend left = 30]
      &
      \\
      & 
      & & V_5  \\
      V_2 \arrow[rruu, blue] \arrow[rrru, blue]  & & & \\
    \end{tikzcd}
    \caption{$\tilde{V} = \{V_1,V_2,V_4,V_5\}$.}
    \label{fig:admg-example-proj-1245}
  \end{subfigure}
  \\
  \begin{subfigure}[b]{0.35\textwidth} \centering
    \begin{tikzcd}
      & & V_4 \arrow[ld, red, leftrightarrow, bend right = 30]
      &
      \\
      & V_3 \arrow[ru, blue] \arrow[rr, blue]  \arrow[rr, red, leftrightarrow, bend left = 30] & & V_5 \\
      V_2 \arrow[ru, blue] \arrow[rrru, blue] & & & \\
    \end{tikzcd}
    \caption{$\tilde{V} = \{V_2,V_3,V_4,V_5\}$.}
    \label{fig:admg-example-proj-2345}
  \end{subfigure}
  \quad
  \begin{subfigure}[b]{0.35\textwidth} \centering
    \begin{tikzcd}
      & & V_4 \arrow[rd, red, leftrightarrow, bend left = 30]
      &
      \\
      & & & V_5 \\
      V_2 \arrow[rruu, blue] \arrow[rrru, blue] & & & \\
    \end{tikzcd}
    \caption{$\tilde{V} = \{V_2,V_4,V_5\}$.}
    \label{fig:admg-example-proj-245}
  \end{subfigure}
  \caption{Marginalization of the graph in \Cref{fig:admg-example-1}
    on different $\tilde{V}$ (showing the trimmed graphs).}
  \label{fig:admg-example-proj}
\end{figure}

For the (trimmed) graph in \Cref{fig:admg-example-proj-12345}, we have
\[
  \textnot V_1 \mconn V_2, ~ \textnot V_2 \mconn V_4 \mid V_1, V_3,
\]
and no other m-separations. For example, we claim that
\[
    V_4 \mconn V_5 \mid L~\text{for all}~L \subseteq \{V_1,V_2,V_3\}.
  \]
  To show this, notice that if $V_3$ is not conditioned on, the path $V_4
  \ldedge V_3 \rdedge
  V_5$ is not ancestrally blocked; this can also be seen from the
  bidirected edge $V_4 \bdedge V_5$ in the marginalized graph for
  $\tilde{V} = \{V_1,V_2,V_4,V_5\}$ in
  \Cref{fig:admg-example-proj-1245}. However, if $V_3$ is conditioned
  on, the path $V_4 \bdedge V_1 \bdedge V_5$ is not
  ancestrally blocked because $V_1$ is an ancestor of $V_3$ and a collider
  in this path. For example, if $L = \{V_2,V_3\}$ there is a
  collider-connected path $V_4 \bdedge V_3 \bdedge V_5$ in
  \Cref{fig:admg-example-proj-2345} which is the marginalization of
  the walk $V_4 \bdedge V_1 \rdedge V_3 \ldedge V_1
  \bdedge V_5$ in \Cref{fig:admg-example-proj-12345}, so $V_4
  \mconn V_5 \mid \{V_2,V_3\}$ by
  \Cref{cor:mconn-colliderconn}.

\section{Gaussian linear system as a graphical statistical model}
\label{sec:conn-gauss-line}

This Section studies the Gaussian linear system in
\Cref{sec:gauss-line-syst} from a graphical
perspective. \textcite{drtonAlgebraicProblemsStructural2018} provided
an excellent review of algebraic problems and known results in such
models. We will not be so ambitious here, but instead focus on some
basic probabilistic properties of Gaussian linear systems and
relate them with graphical concepts introduced in
\Cref{sec:import-types-walks}.

From an abstract point of view, a statistical model is a collection of
probability distributions on a random vector. Let $V$ be a random
vector that takes values in $\mathbb{V} \subseteq \mathbb{R}^d$, and
the largest statistical model we will consider is denoted as
$\mathbb{P}(V)$, the set of all probability distributions of $V$ with
a density function (with respect to a dominating measure on
$\mathbb{V}$, which is usually the Lebesgue measure if the variables
are continuous or the counting measure if the variables if the
variables are discrete). With the possible addition of some regularity
(e.g.\ smoothness) conditions on the density function, this is often
referred to as the \emph{nonparametric model} in the statistics
literature.

First of all, let us introduce a function $\sigma$ on the walk algebra
that makes the connection between Gaussian linear systems and graphs
more explicit. The parameters $\beta$ and $\Lambda$ in a Gaussian
linear system can be more succinctly described as a weight function
$\sigma$ for the edges in its associated graph:
\begin{equation}
  \label{eq:sigma-basic-matrix}
  \beta = \sigma(W[V \rdedge V])\quad\text{and}\quad\Lambda =
  \sigma(W[V \bdedge V]),
\end{equation}
where $\sigma$ is applied entry-wise. This function $\sigma$ can be
extended to the entire walk algebra: the weight of a walk is the
product of the weights of its edges, and the weight of a set of walks
is the sum of their weights. For example,
\[
  \sigma(\{V_1 \bdedge V_4, V_1 \bdedge V_1
  \rdedge V_3 \rdedge V_4\}) = \lambda_{14} + \lambda_{11}
  \beta_{13} \beta_{34},
\]
As a convention, $\sigma(\id) = 1$ for the trivial walk $\id$, so
$\sigma(\Id) = \Id$ where the first identity matrix is a diagonal
matrix of trivial walks and the second is a diagonal matrix of
ones. Let $\GLS(\gG,\sigma)$ denote the set of distributions of $V$ in
a Gaussian linear system with respect to graph $\gG \in \DMG(V)$ and
weight function $\sigma$. If the linear system is non-singular, this
set contains a single distribution $\normal(0, \Sigma)$ where
\[
  \Sigma = \Sigma(\sigma) = \Sigma(\beta,\Lambda) = (\Id - \beta)^{-T}
  \Lambda (\Id - \beta)^{-1}
\]
is given by \eqref{eq:gls-dist}. Let $\GLS_{\textnormal{N}}(\gG)$
denote the set of all non-singular Gaussian linear systems:
\begin{align*}
  \GLS_{\textnormal{N}}(\gG) &= \bigcup_{\sigma~\text{is non-singular}}
  \GLS(\gG, \sigma) \\
  &= \left\{\normal\left(0, \Sigma(\beta,\Lambda)\right):~\text{non-singular $\beta$ and positive
    semi-definite
    $\Lambda$ given by \eqref{eq:sigma-basic-matrix}}\right\}.
\end{align*}
Similarly, let
$\GLS_{\textnormal{PN}}(\gG)$, $\GLS_{\textnormal{S}}(\gG)$, and
$\GLS_{\textnormal{PS}}(\gG)$ denote, respectively, the set of
principally non-singular, stable, and principally stable Gaussian
linear systems. Finally, we add a superscript ``$+$'' to these
notations to indicate the additional requirement that $\Lambda$ is
positive definite. This ensures that the normal distributions are not
degenerate (have a density function with respect to the Lebesgue
measure on $\mathbb{R}^d$), thus $\GLS^+_{\textnormal{N}}(\gG)$ is
indeed a submodel of $\mathbb{P}(V)$.

\subsection{Trek rule and path analysis}
\label{sec:trek-rule-path}

Our first result in this Section justified the definition of treks in
\Cref{sec:arcs}. This follows immediately from matching
\eqref{eq:gls-dist} to \eqref{eq:trek} term by term and using the
following graphical interpretation of the Neumann series.

\begin{proposition} \label{prop:neumann-series}
  Let $\sigma$ be the weight function of a stable linear system and
  $\beta$ be given by \eqref{eq:sigma-basic-matrix}. Then
  \begin{equation*}
    (\Id - \beta)^{-1} = \sum_{q=0}^{\infty} \beta^q = 
    \Id +
    \sigma(W[V \rdpath V]).
  \end{equation*}
\end{proposition}

\begin{theorem} \label{thm:trek-rule}
  For any $\gG \in \DMG(V)$, $\P \in
  \GLS_{\textnormal{N}}(\gG)$, and $V_j,V_k \in V$, we have
  \[
    V_j \notrek V_k \ingraph{\gG} \Longrightarrow \Cov_{\P}(V_j, V_k)
    = 0 \Longleftrightarrow V_j \independent V_k \underdist{\P}.
  \]
  If further $\P \in \GLS_{\textnormal{S}}(\gG,\sigma)$ so $\sigma$ is
  stable, the covariance matrix of $V$ is given by
  \begin{equation}
    \label{eq:trek-rule}
    \Cov_{\P}(V) = \sigma(W[V \trek V \ingraph{\gG}]).
  \end{equation}
\end{theorem}

Equation \eqref{eq:trek-rule} is known as the \emph{trek rule} in the
literature. The term ``trek'' was first used by
\textcite{spirtesCausationPredictionSearch1993} in the context of
directed acyclic graphs in a slightly different way. 
By definition, treks have two (possibly trivial) directed ``legs'' that
are separated by a bidirected edge. Some authors use treks to refer to
the tuple of those three components
\parencite{sullivantTrekSeparationGaussian2010}.

In acyclic graphs, the trek rule admits a much simpler form which only
involves paths in the graph. The result is based on the following
partitions of treks and m-connected paths from $V_j$ to $V_k$ ($V_j
\neq V_k$):
\begin{align*}
  W[V_j \trek V_k] &= P[V_j \trek V_k] + \sum_{V_r \in V} W[V_j \trek
                     V_k~\text{via root}~V_r], \\
  P[V_j \mconnarc V_k] &= P[V_j \trek V_k] + \sum_{V_r \in V} P[V_j \dconnarc
                         V_k~\text{via root}~V_r],
\end{align*}
where the root of a directed arc is defined as the unique vertex on
the walk without an arrowhead pointing to it, that is,
\begin{equation}
  \label{eq:dconn-arc-root}
  P[V_j \dconnarc V_k~\text{via root}~V_r] =
  \begin{cases}
      P[V_j \rdpath V_k],& \text{if}~V_r = V_j, \\
      P[V_j \ldpath V_k],& \text{if}~V_r = V_k, \\
      P[V_j \ldpath V_r \rdpath V_k],&\text{otherwise},
  \end{cases}
\end{equation}
and the root of a trek is defined as its first repeated vertex (if the
trek is not a path) and can be formally defined as
\[
W[V_j \trek V_k~\text{via root}~V_r] = \left\{w_1 \cdot W[V_r \trek
       V_r]
       \cdot w_2: w_1 \cdot w_2 \in P[V_j \dconnarc V_k~\text{via
       root}~V_r]
       \right\}.
\]
By the trek rule, $\Var(V_r) = \sigma(W[V_r \trek V_r])$ for all $V_r
\in V$ if the linear system is stable. This implies the next Theorem
known as \emph{Wright's path analysis}
\parencite{wrightRelativeImportanceHeredity1920,wright34_method_path_coeff}. Note
that for acyclic graphs, any non-singular linear system must be
principally stable, that is,
\[
  \gG \in \ADMG(V) \Longrightarrow \GLS_{\textnormal{N}}(\gG) =
  \GLS_{\textnormal{PS}}(\gG),
\]
because all eigenvalues of $\beta$ must be zero ($\beta$ can be
rearranged into a strictly triangular matrix).

\begin{theorem} \label{thm:path-analysis}
  Consider $\P \in \GLS_{\textnormal{N}}(\gG,\sigma)$ for some $\gG \in
  \ADMG(V)$. Then for any $V_j, V_k
  \in V$, $j \neq k$, their covariance is given by
  \begin{equation*}
      \Cov_{\P}(V_j,V_k) = \sigma(P[V_j \trek V_k \ingraph{\gG}])
      + \sum_{V_r \in V} \sigma(P[V_j
    \dconnarc V_k~\textnormal{via root}~V_r \ingraph{\gG}]) \cdot
      \Var_{\P}(V_r).
  \end{equation*}
\end{theorem}

\subsection{Marginalization of Gaussian linear systems}
\label{sec:marg-gauss-line}

Consider $\P \in \GLS_{\textnormal{N}}(V)$ and $\tilde{V} \subseteq
V$. A natural probabilistic question is: what is the marginal
distribution of $\tilde{V}$? Because $V$ is multivariate normal, so is
its sub-vector $\tilde{V}$. But we can say much more: in general, the
marginal distribution is a Gaussian linear system with respect to the
marginal graph.

To see this, let us denote $U = V \setminus \tilde{V}$ and partition
the linear system as
  \begin{align*}
    \tilde{ V} &= 
                     \beta_{\tilde{V}
                    \tilde{V}}^T \tilde{ V} +  \beta_{U \tilde{V}}^T
                    U +  E_{\tilde{V}}, \\
     U &= 
             \beta_{\tilde{V}
            U}^T \tilde{ V} +  \beta_{U U}^T  U
            +  E_{U}.
  \end{align*}
By eliminating $U$, we obtain
\begin{align}
  \tilde{V} &= \left\{\beta_{\tilde{V} \tilde{V}}^T  +
        \beta_{U \tilde{V}}^T (\Id -
       \beta_{UU})^{-T}  \beta_{\tilde{V}
       U}^T \right\} \tilde{ V} + \left\{ \beta_{U \tilde{V}}^T (\Id -
       \beta_{UU})^{-T}  E_{U} + E_{\tilde{V}}
     \right\} \tag*{} \\
    &= \sigma[\tilde{V} \ldpath \tilde{V} \mid
       \tilde{V}]\, \tilde{V} + \left\{ \sigma[\tilde{V} \ldpath U
       \mid \tilde{V}]\, E_{U} + E_{\tilde{V}} \right\}, \label{eq:eliminate-u}
\end{align}
where the first equality requires the system to be principally
non-singular (so we can invert $\Id - \beta_{UU}$), and the second
equality requires principal stability, so we can use
\Cref{prop:neumann-series} and \eqref{eq:unblocked-directed-walk} to
obtain
\[
(\Id - \beta_{U U})^{-1} = \Id + \sigma(W[U \rdpath U \mid
\tilde{V}]).
\]
From \eqref{eq:eliminate-u}, it is not difficult to prove the
following Theorem which motivates the definition of marginal graphs in
\Cref{sec:marg-direct-mixed}.

\begin{theorem}[Marginalization of linear
  systems] \label{thm:latent-proj-linear-system}
  For any $\gG \in \ADMG(V)$ and $\tilde{V} \subseteq \bm V$, we have
  \[
    \P \in \GLS_{\textnormal{PS}}(\gG,\sigma) \Longrightarrow
    \marg_{\tilde{V}}(\P) \in
    \GLS_{\textnormal{PN}}(\marg_{\tilde{V}}(\gG),
    \marg_{\tilde{V}}(\sigma)),
  \]
  where the weight function $\tilde{\sigma} =
  \marg_{\tilde{V}}(\sigma)$ of the marginal linear system for
  $\tilde{V}$ is generated by
  \begin{align*}
    \tilde{\beta} = \tilde{\sigma}(W[\tilde{V}
    \rdedge \tilde{V} \ingraph{\tildegG}]) = \sigma(W[\tilde{V}
    \rdpath \tilde{V} \mid \tilde{V}
    \ingraph{\gG}]), 
    \\
    \tilde{\Lambda} = \tilde{\sigma}(W[\tilde{V}
    \bdedge \tilde{V} \ingraph{\tildegG}]) = \sigma(W[\tilde{V}
    \trek \tilde{V} \mid \tilde{V}
    \ingraph{\gG}]). 
  \end{align*}
\end{theorem}

From a linear algebra perspective, marginalization is just performing
block Gaussian elimination. In
particular, $(\Id - \tilde{\beta})$ is the Schur complement of $(\Id -
\beta)$. By using the quotient property of Schur complementation, it
is not difficult to show that the marginal of a principally
non-singular system is also principally non-singular:
\[
  \marg_{\tilde{V}}(\GLS_{\textnormal{PN}}(\gG)) \subseteq
  \GLS_{\textnormal{PN}}(\marg_{\tilde{V}}(\gG)).
\]

Note that marginalization does not preserves principal stability in
general. One example where principal stability is
preserved is when $\gG$ is acyclic, so $\beta$ can be rearranged as a
strictly triangular matrix. Similarly, principal stability is preserved when
$\tilde{V}$ is ancestral, so the matrix $(\Id - \beta)$ is block
triangular. Principal stability is also preserved
when all entries of $\beta$ are non-negative. In this case, $(\Id -
\beta)$ is an instance of M-matrices in linear algebra that are closed
under Schur complementation \parencite[p.\
125]{johnsonClosureProperties2005}.

\subsection{Conditional independences}
\label{sec:cond-indep}

As mentioned in \Cref{sec:gauss-line-syst}, conditional independences
in multivariate normal variables can be read off from the inverse
covariance matrix (after possible marginalization). By inverting
\eqref{eq:gls-dist}, we obtain
\[
  \Sigma^{-1} = (\Id - \beta) \Lambda^{-1} (\Id - \beta)^T,
\]
if the linear system is non-singular and $\Lambda$ is positive
definite. Thus, $V_j$ and $V_k$ are conditionally independent given
the rest of the variables if and only if
\begin{equation}
  (\Sigma^{-1})_{jk} = \sum_{V_l,V_m \in V} (\delta_{jl} - \beta_{jl})
  ( \Lambda^{-1})_{lm} (\delta_{km} - \beta_{km})
     = 0. \label{eq:inverse-covariance-decomposition}
\end{equation}
By corresponding the summands in
\eqref{eq:inverse-covariance-decomposition} to walks on the graph, we
see that they all vanish if $V_j$ and $V_k$ are connected by a
sequence of colliders.

\begin{proposition} \label{prop:cond-indep-lsem-smallest}
  Suppose $\P \in \GLS_{\textnormal{N}}^+(\gG)$ for some $\gG \in
  \DMGc(V)$. 
  Then for any $V_j, V_k \in V$, $V_j \neq V_k$, we have
  \[
    \textnot V_j \colliderconn V_k \ingraph{\gG} \Longrightarrow V_j
    \independent V_k
    \mid V \setminus \{V_j,V_k\} \underdist{\P}. 
  \]
\end{proposition}

The key to prove \Cref{prop:cond-indep-lsem-smallest} is the
observation that for $V_l,V_m \in V$, $V_l \neq V_m$,
\begin{equation}
  \label{eq:diffdist-lambda-inv-0}
  \textnot V_l \samedist V_m \ingraph{\gG} \Longrightarrow
  (\Lambda^{-1})_{lm} = 0
\end{equation}
because $\Lambda$ can be rearranged into a block diagonal matrix for
which $V_l$ and $V_m$ belong to different blocks.

The next Theorem shows that m-separation implies conditional
independence in Gaussian linear systems. This non-trivial result is
independently discovered by
\textcite{spirtesDirectedCyclicGraphical1995,kosterMarkovPropertiesNonrecursive1996a},
and a full proof using directed mixed graphs is given in
\textcite{kosterValidityMarkovInterpretation1999}. We will see below
that it is a simple corollary of the results stated above.

\begin{theorem} \label{thm:msep-lsem}
  Suppose $\P \in \GLS_{\textnormal{N}}^+(\gG)$ for some $\gG \in
  \DMGc(V)$.
  Then for all disjoint $J, K, L \subseteq V$, we have
  \begin{equation}
    \label{eq:msep-lsem}
    \textnot J \mconn K \mid L \ingraph{\gG} \Longrightarrow J
    \independent K \mid L \underdist{\P}.
  \end{equation}
\end{theorem}
\begin{proof}
  By definition,
  \[
    \textnot J \mconn K \mid L \Longleftrightarrow
    \textnot V_j \mconn V_k \mid L~\text{for all}~ V_j
    \in J, V_k \in K.
  \]
  Further, because $V$ is multivariate normal with a
  positive definite covariance matrix, it is not difficult to show
  that $J \independent K \mid L$ if and only if $V_j \independent V_k
  \mid L$ for all $V_j \in J$ and $V_k \in K$. Thus it suffices to
  consider $J = \{V_j\}$ and $K = \{V_k\}$. Let $\tilde{\gG} =
  \marg_{\{V_j, V_k\} \cup L}(\gG)$, then
  \begin{align*}
    \textnot V_j \mconn V_k \mid L \ingraph{\gG} \Longleftrightarrow &~ \textnot V_j \colliderconn
                                                             V_k \ingraph{\tildegG}
                                                             \tag*{(\Cref{cor:mconn-colliderconn})}\\
    \Longrightarrow & ~V_j \independent V_k \mid L \underdist{\P},
                      \tag*{(\Cref{prop:cond-indep-lsem-smallest})}
  \end{align*}
  where in the second step we use
  \Cref{thm:latent-proj-linear-system} to verify that $(V_j,V_k,L)$
  follows a non-singular linear system with respect to $\tilde{\gG}$,
  and it is not difficult to show that $\tilde{\Lambda}$ is positive
  definite.
\end{proof}

Equation \eqref{eq:msep-lsem} is called the \emph{global Markov}
property with respect to the graph $\gG$. It also holds for
nonparametric models when the graph is acyclic but not when the graph
is cyclic; see \Cref{sec:discussion} for further discussion. Gaussian
linear systems may also have other equality constraints not in the
form of conditional independence
\parencite{sullivantTrekSeparationGaussian2010,shpitserAcyclicLinearSEMs2018}.

\subsection{Examples}
\label{sec:examples-2}

The Gaussian linear system corresponding to \Cref{fig:admg-example-proj-12345} is
  \begin{align*}
    V_1 &= E_1, \\
    V_2 &= E_2, \\
    V_3 &= \beta_{13} V_1 + \beta_{23} V_2 + E_3, \\
    V_4 &= \beta_{34} V_3 + E_4, \\
    V_5 &= \beta_{25} V_2 + \beta_{35} V_3 + \beta_{45} V_4
          + E_5,
  \end{align*}
where the noise vector $E = (E_1,\dots,E_5)$ is multivariate normal
with mean zero and a covariance matrix of the following form
\[
  \Lambda =
    \begin{pmatrix}
      \lambda_{11} & 0 & 0 & \lambda_{14} & \lambda_{15} \\
      0 & \lambda_{22} & 0 & 0 & 0 \\
      0 & 0 & \lambda_{33} & 0 & 0 \\
      \lambda_{41} & 0 & 0 & \lambda_{44} & \lambda_{45} \\
      \lambda_{51} & 0 & 0 & \lambda_{54} & \lambda_{55}
    \end{pmatrix},
\]
provided that $\Lambda$ is symmetric and positive semi-definite. We
can apply the trek rule to obtain
  \begin{align*}
    \Cov(V_1,V_4) &= \Lambda_{14} + \Lambda_{11} \beta_{13}
                    \beta_{34}, \\
    \Var(V_3) &= \Lambda_{33} + \Lambda_{11} \beta_{13}^2 +
                \Lambda_{22} \beta_{23}^2, \\
    \Cov(V_3,V_4) &= \Lambda_{33} \beta_{34} + \Lambda_{14} \beta_{13} +
                    \Lambda_{11} \beta_{13}^2 \beta_{34} +
                    \Lambda_{22} \beta_{23}^2 \beta_{34} \\
                    &= \Lambda_{14} \beta_{13} + \Var(V_3) \beta_{34}.
  \end{align*}
  The last equation can also be obtained using path analysis:
  \[
    \Cov(V_3, V_4) = \sigma(V_3 \ldedge V_1 \bdedge V_4) +
    \sigma(V_3 \rdedge V_4) \Var(V_3) = \Lambda_{14} \beta_{13} +
    \Var(V_3) \beta_{34}.
  \]

The marginal linear system on $\tilde{V} = (V_1,V_2,V_4,V_5)$ can be
obtained by plugging the equation for $V_3$ into the other equations:
  \begin{align*}
    V_1 &= 
          E_1, \\
    V_2 &= 
          E_2, \\
    V_4 &= 
          \beta_{13} \beta_{34}
          V_1 + \beta_{23} \beta_{34} V_2 + (\beta_{34} E_3 + E_4), \\
    V_5 &= 
          \beta_{13} \beta_{35} V_1
          + (\beta_{25} + \beta_{23} \beta_{35}) V_2 + (\beta_{35} E_3
          + E_5).
  \end{align*}
This is a linear system with respect to the marginal graph shown in
\Cref{fig:admg-example-proj-1245}.

\Cref{thm:msep-lsem} shows that this Gaussian linear system has two
conditional independences: $V_1 \independent V_2$, $V_2 \independent
V_4 \mid V_1, V_3$.

\section{Further examples}
\label{sec:further-examples}

In this Section we give some further examples to demonstrate how the
matrix algebra introduced in
\Cref{sec:matr-algebra-direct,sec:import-types-walks} empowers us to
solve non-trivial problems about graphical statistical models.

\subsection{Confounder adjustment}
\label{sec:conf-adjustm}

Confounding is a central concept in causal inference. Consider a
Gaussian linear system with respect to $\gG \in \ADMGc(V)$ and disjoint
$\{V_j\},\{V_k\}, L \subset V$. The equations in the linear system
\eqref{eq:linear-system-matrix} can be interpreted causally, that is,
we may assume that when $V_j$ is set to $v_j$ in an intervention, the
rest of the variables are determined by substituting $V_j$ with $v_j$
in all equations other than the one for $V_j$. This can be formally
described as a statistical model on the ``potential outcomes'' of the
variables using the so-called \emph{non-parametric structural equation
  model}
\parencite{pearlCausalityModelsReasoning2000,richardson2013single}. In
this case, the \emph{total causal effect} of $V_j$ on $V_k$ can be
defined as $\sigma(P[V_j \rdpath V_k])$, which is equal to the
derivative of $V_k(v_j)$ with respect to $v_j$ if $V_k(v_j)$ denote
the value of $V_k$ under the intervention described above.

The total causal effect is often estimated by the coefficient of $V_j$ in the
linear regression of $V_k$ on $V_j$ and $L$, which can be written as
\begin{equation}
  \label{eq:ols}
  \gamma(V_k, V_j \mid L) = \frac{\Cov(V_k,V_j \mid L)}{\Var(V_j \mid
    L)} = -
  \frac{(\tilde{\Sigma}^{-1})_{jk}}{(\tilde{\Sigma}^{-1})_{kk}},
\end{equation}
where $\tilde{\Sigma}$ is the marginal covariance matrix of $(V_j,
V_k, L)$. We can investigate how this regression coefficient is
related to the total causal effect by expanding the numerator and
denominator in \eqref{eq:ols}.

\begin{proposition} \label{prop:ols-gtilde}
  Consider $\gG \in \ADMGc(V)$, $\P \in \GLS^+_{\textnormal{N}}(\gG)$,
  and disjoint $\{V_j\},\{V_k\},L \subset V$. Denote $\tilde{V} =
  \{V_j, V_k\} \cup L$, $\tilde{\gG} = \marg_{\tilde{V}}(\gG)$, and
  $\tilde{\sigma} = \marg_{\tilde{V}}(\gG)$. Suppose the following
  conditions are satisfied:
  \begin{enumerate}
  \item
    Only $V_j \rdedge V_k$ can be a path from
    $V_j$ to $V_k$ in $\tilde{\gG}$ that is not blocked by $L$:
    \[
      P[V_j \mconn V_k \mid L \ingraph{\tildegG}] \subseteq \{V_j
      \rdedge V_k\};
    \]
  \item
    $L$ contains no descendant of $V_k$ in $\tilde{\gG}$:
    $\textnot V_k \rdpath L \ingraph{\tildegG}$.
  \end{enumerate}
  Then we have $\gamma(V_k, V_j \mid L) = \tilde{\sigma}(P[V_j \rdedge
  V_k \ingraph{\tildegG}])$.
\end{proposition}
\begin{proof}
  Denote $\tilde{\beta} = \tilde{\sigma}(P[V \rdedge V
  \ingraph{\tildegG}])$ and $\tilde{\Lambda} = \tilde{\sigma}(P[V
  \bdedge V \ingraph{\tildegG}])$. The first condition implies that
  $\textnot V_j \ldedge V_k \ingraph{\tildegG}$, so
  $\tilde{\beta}_{kj} = 0$.

  By applying the expansion in
  \eqref{eq:inverse-covariance-decomposition} to \eqref{eq:ols}, we
  obtain
  \begin{equation}
    \label{eq:gamma-expansion}
        \gamma(V_k, V_j \mid L) = - \frac{\sum_{V_l,V_m \in \tilde{V}}
      (\delta_{jl} - \tilde{\beta}_{jl}) (\tilde{\Lambda}^{-1})_{lm}
      (\delta_{km} - \tilde{\beta}_{km})}{\sum_{V_l,V_m \in \tilde{V}}
      (\delta_{kl} - \tilde{\beta}_{kl}) (\tilde{\Lambda}^{-1})_{lm}
      (\delta_{km} - \tilde{\beta}_{km})} := -
    \frac{\text{Num}}{\text{Denom}}.
  \end{equation}
  It can be shown that all terms for which $V_l \in L$ must be $0$,
  because any such non-zero term implies, by
  \eqref{eq:diffdist-lambda-inv-0}, a path from $V_j$ to $V_k$ not
  blocked by $L$ that is not a directed edge. 
  The second condition
  implies that $\tilde{\beta}_{km} = 0$ for all $V_m \in L$. Thus,
  \[
    \text{Num} = - \sum_{V_m \in \tilde{V}} \tilde{\beta}_{jk}
    (\tilde{\Lambda}^{-1})_{km} (\delta_{km} - \tilde{\beta}_{km}) = -
    \tilde{\beta}_{jk} (\tilde{\Lambda}^{-1})_{kk}.
  \]
  By investigating the terms in the denominator in a similar way, we
  obtain $\text{Denom} = (\tilde{\Lambda}^{-1})_{kk}$. The conclusion
  immediately follows.
\end{proof}

The conditions in \Cref{prop:ols-gtilde} are stated using the marginal
graph $\tilde{\gG}$. The next result translates this to the original
graph $\gG$.

\begin{theorem} \label{thm:ols-g}
In the setting of \Cref{prop:ols-gtilde}, the two conditions in
\Cref{prop:ols-gtilde} hold if and only if the following are all true:
\begin{enumerate}
\item All paths from $V_j$ to $V_k$ in $\gG$
  that are not ancestrally blocked given $L$ must be right-directed:
  \[
    P[V_j \mconn V_k \mid_a L \ingraph{\gG}] \subseteq P[V_j \rdpath
    V_k \ingraph{\gG}].
  \]
\item $L$ contains no descendant of $V_k$ in $\gG$: $\textnot V_k
  \rdpath L \ingraph{\gG}$.
\item $L$ contains no descendant of non-endpoint vertices on any
  directed path from $V_j$ to $V_k$ given $L$: there does not exist
  $V_l \in V \setminus \tilde{V}$ such that
  $V_j \rdpath V_l \rdpath V_k \mid L \ingraph{\gG}$ and $V_l \rdpath
  L \ingraph{\gG}$.
\end{enumerate}
If these conditions are met, we have $\gamma(V_k, V_j \mid L) =
\sigma(P[V_j \rdpath V_k \mid L \ingraph{\gG}])$.
\end{theorem}
\begin{proof}

  Note that, by definition and acyclicity, $P[V_j \rdpath V_k \mid_a L
  \ingraph{\gG}]$ is the pre-image of $V_j \rdedge V_k$ under
  $\marg_{\tilde{V}}$. The only non-trivial fact is that the first
  condition in \Cref{prop:ols-gtilde} is implied by the conditions in
  this Theorem. If that is not true, consider any
  \[
    \tilde{w} \in P[V_j \mconn V_k \mid_a L \ingraph{\tildegG}]
    \setminus \{V_j \rdedge V_k\}.
  \]
  Let $w$ be a pre-image of $\tilde{w}$ under $\marg_{\tilde{V}}$, so
  $\marg_{\tilde{V}}(w) = \tilde{w}$. Because $\tilde{w}$ is a path,
  $V_j$ and $V_k$ only appear once in $w$. This shows that
  \[
    w \in W[V_j \mconn V_k \mid \tilde{V} \ingraph{\gG}] \setminus W[V_j
    \rdpath V_k \mid \tilde{V} \ingraph{\gG}].
  \]
  By taking necessary
  shortcuts to remove duplicated vertices in $w$, we obtain a path
  $w'$ from $V_j$ to $V_k$ that is not ancestrally blocked by $L$. By
  the first condition in this Theorem, $w' \in P[V_j \rdpath V_k \mid L
  \ingraph{\gG}]$. This is only possible if there exists $V_l \in V$
  such that
  \[
    w \in W[V_j \rdpath V_l \mconn V_l \rdpath V_k \mid \tilde{V}
    \ingraph{\gG}].
  \]
  Because the last instance of $V_l$ in $w$ is not a collider, we know
  $V_l \not \in \tilde{V}$ and the first instance of $V_l$ in $w$ is
  also not a collider. This shows that $V_j \rdpath V_l \rdpath V_k
  \mid L$ and $V_l \rdpath L \ingraph{\gG}$, which contradicts the
  third condition in the Theorem statement.
\end{proof}

Given the conditions in \Cref{thm:ols-g} and the additional
requirement that $L$ contains no vertices on a directed path from
$V_j$ to $V_k$, we see that $\gamma(V_k, V_j \mid L)$ identifies the
total causal effect $\sigma(P[V_j \rdpath V_k])$. This is exactly the
criterion in \textcite{shpitserValidityCovariateAdjustment2010} which
implies the backdoor criterion in
\textcite{pearlCausalDiagramsEmpirical1995}. \textcite{shpitserValidityCovariateAdjustment2010}
consider the nonparametric identification problem with sets of cause
and effect variables and showed the criterion above is sound and
complete for identifying the total causal effect via a nonparametric
confounder adjustment formula that generalizes \eqref{eq:ols}. The
consideration of controlled direct effect (not mediated by $L$) in
\Cref{thm:ols-g} is possibly novel.

The conditions in \Cref{prop:ols-gtilde} and \Cref{thm:ols-g} are
asymmetric in the potential cause $V_j$ and effect
$V_k$. Under the assumption that $L$ contains no descendant of
$V_j$ and $V_k$, that is
\begin{equation}
  \label{eq:no-descendant}
  \textnot  \{V_j, V_k\} \rdpath L \ingraph{\gG},
\end{equation}
\textcite{guoConfounderSelectionIterative2023} propose to
characterize ``no confounding'' graphically using the following
symmetric criterion
\begin{equation}
  \label{eq:no-confounding}
P[V_j \confpath V_k \mid L \ingraph{\gG}] = \emptyset,
\end{equation}
and use this to develop an iterative algorithm for confounder
selection. Equation \eqref{eq:no-confounding} means that there is no
path from $V_j$ to $V_k$ not ancestrally blocked by $L$ that is
consisted of one or several ``confounding arcs'' $\confarc$. Formally,
define
\[
  P[V \confpath V \mid X] = P[V \confpath V] \cap W[V \mconn V \mid L],
\]
where
\begin{equation}
  \label{eq:confarc}
  W[V \confarc V] = W[V \mconnarc V] \setminus W[V \rdpath V]
  \setminus W[V \ldpath V],~P[V \confarc V] = W[V \confarc V] \cap
  \mathcal{P}_{\gG},
\end{equation}
\begin{equation}
    \label{eq:confpath}
W[V \confpath V] = \sum_{q=1}^{\infty} (W[V \confarc V])^q,~P[V
\confpath V] = W[V \confpath V] \cap \mathcal{P}_{\gG}.
\end{equation}
Consider the following result.

\begin{theorem} \label{thm:symmetric-backdoor}
  Consider $\gG \in \ADMGc(V)$, $\P \in \GLS^+_{\textnormal{N}}(\gG)$,
  and disjoint $\{V_j\},\{V_k\},L \subset V$. Suppose
  \eqref{eq:no-descendant} and \eqref{eq:no-confounding} are
  satisfied. Then
  \[
    V_k \rdpath V_j \ingraph{\gG} \Longrightarrow \gamma(V_k,V_j \mid
    L) = \sigma(P[V_j \rdpath V_k \ingraph{\gG}]).
  \]
\end{theorem}
\begin{proof}
  Consider the marginal graph $\tilde{\gG}$ and the marginal linear
  system as in \Cref{prop:ols-gtilde}. It can be shown that
  \eqref{eq:no-confounding} implies
\[
  \textnot V_j \samedist V_k \ingraph{\tildegG},
\]
so by \eqref{eq:diffdist-lambda-inv-0} we know
$(\tilde{\Lambda}^{-1})_{jk} = 0$.
  By using \eqref{eq:no-descendant} and the acyclicity of $\gG$ (so
  $\tilde{\beta}_{jk} \tilde{\beta}_{kj} = 0$),
  \eqref{eq:gamma-expansion} can be immediately simplified as
  \begin{equation}
    \label{eq:ols-coef-no-direction}
      \gamma(V_k, V_j \mid L) = \frac{(\tilde{\Lambda}^{-1})_{jk} -
    (\tilde{\Lambda}^{-1})_{jj} \tilde{\beta}_{kj} -
    (\tilde{\Lambda}^{-1})_{kk}
    \tilde{\beta}_{jk}}{(\tilde{\Lambda}^{-1})_{kk} - 2
    (\tilde{\Lambda}^{-1})_{jk} \tilde{\beta}_{kj} +
    (\tilde{\Lambda}^{-1})_{jj} \tilde{\beta}_{kj}^2} = -
  \frac{(\tilde{\Lambda}^{-1})_{jj} \tilde{\beta}_{kj} +
    (\tilde{\Lambda}^{-1})_{kk}
    \tilde{\beta}_{jk}}{(\tilde{\Lambda}^{-1})_{kk} +
    (\tilde{\Lambda}^{-1})_{jj} \tilde{\beta}_{kj}^2}.
  \end{equation}
By acyclicity of $\gG$, the assumption $V_j \rdpath V_k \ingraph{\gG}$
implies that $\textnot V_k \rdpath V_j \ingraph{\gG}$ and thus
$\tilde{\beta}_{kj} = 0$. Because $L$ contains no descendant of $V_j$,
we obtain
\[
  \gamma(V_k, V_j \mid L) = \tilde{\beta}_{jk} = \tilde{\sigma}(W[V_j \rdpath
  V_k \ingraph{\tildegG}]) = \sigma(W[V_j \rdpath V_k \mid \tilde{V}
  \ingraph{\gG}]) = \sigma(P[V_j \rdpath V_k \ingraph{\gG}])
\]
as stated in the Theorem.
\end{proof}

Equation \eqref{eq:ols-coef-no-direction} in the proof above shows
that the regression coefficient $\gamma(V_k, V_j \mid L) \neq 0$
whenever $\tilde{\beta}_{jk}$ or $\tilde{\beta}_{kj}$ is
nonzero. Thus, the symmetric criteria in \eqref{eq:no-descendant} and
\eqref{eq:no-confounding} are sufficient to ``remove confounding'',
and interpreting the regression coefficient of $V_k$ on $V_j$ or $V_j$
on $V_k$ (given $L$) as a causal effect entirely rests on the presumed
direction of causality.

\subsection{Augmentation criterion for m-separation}
\label{sec:augm-crit-m}

There are two basic types of graphical statistical models: directed
and undirected. To see how these models are related, let us introduce
some additional notation. Let $\UG(V)$ denote the collection of all
simple undirected graphs with vertex set $V$; specifically, $\UG(V)$
contains all graphs $\gG = (V, \sE)$ such that $\sE \subseteq V \times
V$, $(V_j, V_j) \not \in \sE$,
and $(V_j,V_k) \in \sE$ implies that $(V_k,V_j) \in \sE$ for all
$V_j,V_k \in V$. This definition is not different from a bidirected
graph besides the requirement of no self-loops, but the semantics of
undirected and bidirected graphs are different. Specifically, for
$\gG \in \UG(V)$ and disjoint subsets $J, K, L \subset V$, we say
$L$ \emph{separate} $J$ and $K$ in $\gG$ and write
\[
  \textnot J \uconn K \mid L \ingraph{\gG},
\]
if every path from a vertex in $J$ to a vertex in $K$ contains an
non-endpoint in $L$.

Consider the following \emph{augmentation} map from directed mixed
graphs to undirected graphs:
\begin{align*}
  \augg: \DMGc(V) &\to \UG(V), \\
         \gG &\mapsto \augg(\gG),
\end{align*}
where $\augg(\gG)$ is an undirected graph with the same vertex set $V$
such that
\begin{equation}
  \label{eq:augmented-edge}
    V_j \udedge V_k \ingraph{\augg(\gG)} \Longleftrightarrow V_j
  \colliderconn V_k \ingraph{\gG}~\text{for all}~V_j, V_k \in V, V_j
  \neq V_k.
\end{equation}
That is, $V_j$ and $V_k$ are adjacent in the augmentation graph if and
only if they are connected by a sequence of colliders in the original
graph.  When this map is restricted to (canonically) directed graphs,
this is known as \emph{moralization} in the literature because it
connects any two parents with the same child. In that case, the
augmentation map is directly motivated by the factorization properties
associated with directed acyclic graphs and undirected graphs.

The next key result relates separation in directed mixed graphs with
separation in the augmented undirected graph. This result was
originally stated for directed acyclic graphs in
\textcite[Proposition 3]{lauritzenIndependencePropertiesDirected1990},
but their proof has gaps. \textcite[Theorem
1]{richardson03_markov_proper_acycl_direc_mixed_graph} first stated
and proved this results for directed mixed graphs. Below we give a
more ``visual'' proof using the matrix algebra introduced above.

\begin{proposition} \label{prop:mconn-uconn}
For any $\gG \in \DMGc(V)$ and disjoint $J, K, L \subset V$, we have
\[
  J \mconn K \mid L \ingraph{\gG} \Longleftrightarrow J \uconn K \mid L
  \ingraph{\augg \circ \marg_{\tilde{V}}(\gG)},
\]
where $\tilde{V} = \overline{\an}(J \cup K \cup L)$.
\end{proposition}
\begin{proof}
  By Corollary \ref{cor:marg-preserve-m-connection}, m-separation in
  $\gG$ is equivalent to m-separation in $\marg_{\tilde{V}}(\gG)$.
  It is obvious that $\tilde{V}$ is ancestral, so the marginal graph
  $\marg_{\tilde{V}}(\gG)$ is simply the subgraph of $\gG$ restricted
  to $\tilde{V}$. So, without loss of generality, we can assume
  \[V = \tilde{V} = \overline{\an}(J \cup K \cup L).\]

  $\Longrightarrow$ direction: consider $w \in W[J
  \mconn K \mid L \ingraph{\gG}]$. Whenever $w$ contains a collisder,
  it can always be separated by segments of one or multiple
  consecutive colliders. By taking shortcuts that bypass those
  colliders, we obtain a walk
  from a vertex in $J$ to a vertex in $V_k$ in the augmented
  undirected graph whose non-endpoints are not in $L$ (because the
  non-endpoints are non-colliders in $w$).
  For example, if $w$ is the following walk in $\gG$
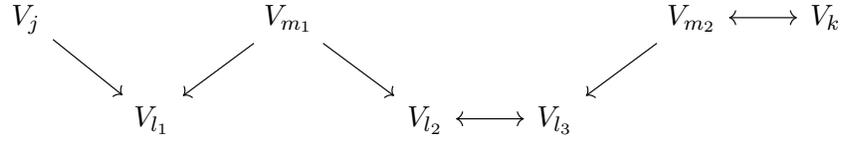
  \[
    \begin{tikzcd}
      V_j \arrow[rd]  & & V_{m_1} \arrow[ld]
      \arrow[rd]  & & & V_{m_2} \arrow[ld]
      \arrow[r, leftrightarrow] & V_k \\
      & V_{l_1} & & V_{l_2} \arrow[r, leftrightarrow] & V_{l_3} &
    \end{tikzcd}
  \]
  the corresponding walk in the augmented undirected graph is $V_j
  \udedge V_{m_1} \udedge V_{m_2} \udedge V_k$. By further trimming
  duplicated non-endpoints in this walk (if they exist), we can find a
  path in $P[J \uconn K \mid L \ingraph{\augg(\gG)}]$.

  $\Longleftarrow$ direction: consider a path
  \[
    V_{m_0} \udedge V_{m_1} \udedge \dots \udedge V_{m_q},~q \geq 1
  \]
  in $\augg(\gG)$ from $J$ to $K$ (so $V_{m_0} = V_j \in J$ and
  $V_{m_q} = V_k \in
  K$) that does not pass through $L$. By \eqref{eq:augmented-edge}, we have
  \[
    V_{m_r} \udedge V_{m_{r+1}} \ingraph{\augg(\gG)} \Longleftrightarrow
    V_{m_r} \colliderconn V_{m_{r+1}} \ingraph{\gG}.
  \]
  Therefore, there exists a walk $w$ from $V_j$ to $V_k$ in
  $\gG$ such that all non-collider on the walk are not in
  $L$. Consider any collider $V_l$ on this walk that is not already in
  $L$, so $V_l \not \in L$ and
  \[
    w = w_1 \cdot w_2,~\text{where}~w_1 \in W[V_j \halfsquigfull \ast
    \fullsquigfull V_l]~\text{and}~w_2 \in W[V_l \fullsquigfull \ast
    \fullsquighalf V_k].
  \]
  Let $w'$ be a walk obtained as follows:
  \begin{enumerate}
  \item If $V_l \in J$, then $w' = w_2$.
  \item If $V_l \in K$, then $w' = w_1$.
  \item If $V_l \not \in J \cup K \cup L$ but $V_l \rdpath J \mid J
    \cup K \cup L$, then
    \[
      w'= w_1' \cdot w_2~\text{for some}~w_1'
      \in W[J \ldpath V_l \mid J \cup K \cup L].
    \]
  \item If $V_l \not \in J \cup K \cup L$ but $V_l \rdpath K \mid J
    \cup K \cup L$, then
    \[
      w'= w_1 \cdot w_2'~\text{for some}~w_2'
      \in W[V_l \rdpath K \mid J \cup K \cup L].
    \]
  \item If $V_l \not \in J \cup K \cup L$ but $V_l \rdpath L \mid J
    \cup K \cup L$, then
    \[
      w' = w_1 \cdot w_3' \cdot (w_3')^T \cdot
    w_2~\text{for some}~w_3' \in W[V_l \rdpath L \mid J \cup K
    \cup L].
    \]
  \end{enumerate}
  Exactly one of the above five scenarios is true because $J,K,L$ are
  disjoint and $V = \overline{\an}(J \cup K \cup L)$.
  It is easy to see that $w'$ is still a walk from $J$ to $K$, all
  non-colliders in $w'$ are not in $L$, and $w'$ has one fewer
  collider than $w$ that is not in $L$. By repeating this operation,
  we eventually obtain a walk in $W[J \mconn K \mid L \ingraph{\gG}]$.
\end{proof}

\section{Discussion}
\label{sec:discussion}

Perhaps the most useful aspect of the matrix algebra introduced here
is its ability to visualize complex graphical concepts. Many examples
are given above: m- and d-connections,
collider-connected walks, confounding paths, etc. Some further
examples are $\delta$-connection
\parencite{didelezGraphicalModelsMarked2008}, which can be expressed
as $\deltaconn$, and $\mu$-connection
\parencite{mogensenMarkovEquivalenceMarginalized2020}, which can be
expressed as $\muconn$. Compared to d/m-connections, they require the
last edge to have an arrowhead pointing to the end-point and thus
reflect ``predictability for the future'' in considering local
dependence (aka Granger causality) in Markov processes.

As mentioned in the Introduction, many results for nonparametric
graphical models have their origins in Gaussian linear systems. This
can be observed in several places of the present article:
\begin{itemize}
\item \Cref{thm:msep-lsem} shows that Gaussian linear systems are
  global Markov (meaning m-separation in the graph implies conditional
  independence between random variables). The same is true in
  nonparametric DAG models in the sense that any distribution that
  factorizes according to the DAG are also global Markov with respect
  to the DAG. This basically follows from \Cref{prop:mconn-uconn} and
  the Hammersley-Clifford theorem for undirected graphs; see
  \textcite[section 3.2.2]{lauritzenGraphicalModels1996}. However, the
  story is much more complicated for nonparametric ADMG models
  \parencite[see
  e.g.][]{richardsonNestedMarkovProperties2023}. Further, a stronger
  notion of graph separation is needed for nonparametric models when
  the graph is cyclic
  \parencite{spirtesDirectedCyclicGraphical1995,bongersFoundationsStructuralCausal2021}.
\item \Cref{thm:latent-proj-linear-system} shows that the marginal of
  a Gaussian linear system is still a Gaussian linear
  system. Likewise, if $\P$ is global Markov with respect to a
  directed mixed graph, then its marginal is global Markov with
  respect to the corresponding marginal graph. This basically follows
  from \Cref{cor:marg-preserve-m-connection}.
\item As mentioned below \Cref{thm:ols-g}, the conditions there for
  Gaussian linear systems are exactly the same as the generalized
  backdoor criterion for the identification of total causal effect in
  nonparametric causal models
  \parencite{shpitserValidityCovariateAdjustment2010}.
\end{itemize}
As mentioned in the Introduction, there are many more connections in
the literature. The exact relation between algebraic properties of
Gaussian linear systems and nonparametric graphical models require
further investigation.

\section*{Acknowledgement}

The author thanks Richard Guo, Wenjie Hu, and Thomas Richardson for
helpful discussion. This research is in part supported by the
Engineering and Physical Sciences Research Council under EP/V049968/1.

  \printbibliography

@book{bollenStructuralEquationsLatent1989,
  title = {Structural {{Equations}} with {{Latent Variables}}},
  author = {Bollen, Kenneth A.},
  year = {1989},
  month = apr,
  edition = {1},
  publisher = {Wiley},
  doi = {10.1002/9781118619179},
  urldate = {2024-01-29},
  isbn = {978-0-471-01171-2 978-1-118-61917-9},
  langid = {english}
}

@article{bongersFoundationsStructuralCausal2021,
  title = {Foundations of Structural Causal Models with Cycles and Latent Variables},
  author = {Bongers, Stephan and Forr{\'e}, Patrick and Peters, Jonas and Mooij, Joris M.},
  year = {2021},
  month = oct,
  journal = {The Annals of Statistics},
  volume = {49},
  number = {5},
  pages = {2885--2915},
  publisher = {Institute of Mathematical Statistics},
  issn = {0090-5364, 2168-8966},
  doi = {10.1214/21-AOS2064},
  urldate = {2024-07-22}
}

@article{didelezGraphicalModelsMarked2008,
  title = {Graphical Models for Marked Point Processes Based on Local Independence},
  author = {Didelez, Vanessa},
  year = {2008},
  journal = {Journal of the Royal Statistical Society: Series B (Statistical Methodology)},
  volume = {70},
  number = {1},
  pages = {245--264},
  issn = {1467-9868},
  doi = {10.1111/j.1467-9868.2007.00634.x},
  urldate = {2023-07-04},
  langid = {english}
}

@incollection{drtonAlgebraicProblemsStructural2018,
  title = {Algebraic Problems in Structural Equation Modeling},
  booktitle = {The 50th {{Anniversary}} of {{Gr{\"o}bner Bases}}},
  author = {Drton, Mathias},
  year = {2018},
  month = jan,
  volume = {77},
  pages = {35--87},
  publisher = {Mathematical Society of Japan},
  doi = {10.2969/aspm/07710035},
  urldate = {2024-04-22}
}

@article{drtonGlobalIdentifiabilityLinear2011,
  title = {Global Identifiability of Linear Structural Equation Models},
  author = {Drton, Mathias and Foygel, Rina and Sullivant, Seth},
  year = {2011},
  month = apr,
  journal = {The Annals of Statistics},
  volume = {39},
  number = {2},
  pages = {865--886},
  publisher = {Institute of Mathematical Statistics},
  issn = {0090-5364, 2168-8966},
  doi = {10.1214/10-AOS859},
  urldate = {2022-11-09}
}

@book{duncanIntroductionStructuralEquation1975,
  title = {Introduction to {{Structural Equation Models}}},
  author = {Duncan, Otis Dudley},
  year = {1975},
  month = apr,
  publisher = {Elsevier Science},
  googlebooks = {7AK2AAAAIAAJ},
  isbn = {978-0-12-224150-5},
  langid = {english}
}

@book{gondranGraphsDioidsSemirings2008,
  title = {Graphs, {{Dioids}} and {{Semirings}}},
  author = {Gondran, Michel and Minoux, Michel},
  year = {2008},
  series = {Operations {{Research}}/{{Computer Science Interfaces}}},
  volume = {41},
  publisher = {Springer US},
  address = {Boston, MA},
  doi = {10.1007/978-0-387-75450-5},
  urldate = {2023-12-15},
  isbn = {978-0-387-75449-9},
  langid = {english}
}

@misc{guoConfounderSelectionIterative2023,
  title = {Confounder Selection via Iterative Graph Expansion},
  author = {Guo, F. Richard and Zhao, Qingyuan},
  year = {2023},
  month = oct,
  number = {arXiv:2309.06053},
  eprint = {2309.06053},
  primaryclass = {math, stat},
  publisher = {arXiv},
  doi = {10.48550/arXiv.2309.06053},
  archiveprefix = {arxiv}
}

@article{guoVariableEliminationGraph2023,
  title = {Variable Elimination, Graph Reduction and the Efficient g-Formula},
  author = {Guo, F Richard and Perkovi{\'c}, Emilija and Rotnitzky, Andrea},
  year = {2023},
  month = sep,
  journal = {Biometrika},
  volume = {110},
  number = {3},
  pages = {739--761},
  issn = {1464-3510},
  doi = {10.1093/biomet/asac062},
  urldate = {2024-01-29}
}

@article{haavelmoStatisticalImplicationsSystem1943,
  title = {The {{Statistical Implications}} of a {{System}} of {{Simultaneous Equations}}},
  author = {Haavelmo, Trygve},
  year = {1943},
  journal = {Econometrica},
  volume = {11},
  number = {1},
  eprint = {1905714},
  eprinttype = {jstor},
  pages = {1--12},
  publisher = {[Wiley, Econometric Society]},
  issn = {0012-9682},
  doi = {10.2307/1905714},
  urldate = {2024-01-29}
}

@article{henckelGraphicalCriteriaEfficient2022,
  title = {Graphical Criteria for Efficient Total Effect Estimation via Adjustment in Causal Linear Models},
  author = {Henckel, Leonard and Perkovi{\'c}, Emilija and Maathuis, Marloes H.},
  year = {2022},
  journal = {Journal of the Royal Statistical Society: Series B (Statistical Methodology)},
  volume = {84},
  number = {2},
  pages = {579--599},
  issn = {1467-9868},
  doi = {10.1111/rssb.12451},
  urldate = {2024-01-29},
  copyright = {{\copyright} 2022 The Authors. Journal of the Royal Statistical Society: Series B (StatisticalMethodology) published by John Wiley \& Sons Ltd on behalf of Royal Statistical Society.},
  langid = {english}
}

@incollection{johnsonClosureProperties2005,
  title = {Closure {{Properties}}},
  booktitle = {The {{Schur Complement}} and {{Its Applications}}},
  author = {Johnson, Charles R. and Smith, Ronald L.},
  editor = {Zhang, Fuzhen},
  year = {2005},
  pages = {111--136},
  publisher = {Springer US},
  address = {Boston, MA},
  doi = {10.1007/0-387-24273-2_5},
  urldate = {2024-04-15},
  isbn = {978-0-387-24273-6},
  langid = {english}
}

@book{joreskogAdvancesFactorAnalysis1979,
  title = {Advances in {{Factor Analysis}} and {{Structural Equation Models}}},
  author = {J{\"o}reskog, K. G. and S{\"o}rbom, Dag},
  year = {1979},
  publisher = {Abt Books},
  googlebooks = {bCe2AAAAIAAJ},
  isbn = {978-0-89011-535-0},
  langid = {english}
}

@article{kosterMarkovPropertiesNonrecursive1996a,
  title = {Markov Properties of Nonrecursive Causal Models},
  author = {Koster, J. T. A.},
  year = {1996},
  month = oct,
  journal = {The Annals of Statistics},
  volume = {24},
  number = {5},
  pages = {2148--2177},
  publisher = {Institute of Mathematical Statistics},
  issn = {0090-5364, 2168-8966},
  doi = {10.1214/aos/1069362315},
  urldate = {2024-07-18}
}

@article{kosterValidityMarkovInterpretation1999,
  title = {On the {{Validity}} of the {{Markov Interpretation}} of {{Path Diagrams}} of {{Gaussian Structural Equations Systems}} with {{Correlated Errors}}},
  author = {Koster, Jan T. A.},
  year = {1999},
  journal = {Scandinavian Journal of Statistics},
  volume = {26},
  number = {3},
  pages = {413--431},
  issn = {1467-9469},
  doi = {10.1111/1467-9469.00157},
  urldate = {2023-01-05},
  langid = {english}
}

@article{kurokiMeasurementBiasEffect2014,
  title = {Measurement Bias and Effect Restoration in Causal Inference},
  author = {Kuroki, Manabu and Pearl, Judea},
  year = {2014},
  month = jun,
  journal = {Biometrika},
  volume = {101},
  number = {2},
  pages = {423--437},
  issn = {0006-3444},
  doi = {10.1093/biomet/ast066},
  urldate = {2024-01-29}
}

@book{lauritzenGraphicalModels1996,
  title = {Graphical {{Models}}},
  author = {Lauritzen, Steffen L.},
  year = {1996},
  series = {Oxford {{Statistical Science Series}}},
  publisher = {Clarendon Press},
  address = {Oxford},
  urldate = {2023-02-20}
}

@article{lauritzenIndependencePropertiesDirected1990,
  title = {Independence Properties of Directed Markov Fields},
  author = {Lauritzen, S. L. and Dawid, A. P. and Larsen, B. N. and Leimer, H.-G.},
  year = {1990},
  journal = {Networks},
  volume = {20},
  number = {5},
  pages = {491--505},
  issn = {1097-0037},
  doi = {10.1002/net.3230200503},
  urldate = {2024-07-21},
  copyright = {Copyright {\copyright} 1990 Wiley Periodicals, Inc., A Wiley Company},
  langid = {english}
}

@book{lothaireAppliedCombinatoricsWords2005,
  title = {Applied {{Combinatorics}} on {{Words}}},
  author = {Lothaire, M.},
  year = {2005},
  series = {Encyclopedia of {{Mathematics}} and Its {{Applications}}},
  publisher = {Cambridge University Press},
  address = {Cambridge},
  doi = {10.1017/CBO9781107341005},
  urldate = {2024-07-17},
  isbn = {978-0-521-84802-2}
}

@article{miaoIdentifyingCausalEffects2018,
  title = {Identifying Causal Effects with Proxy Variables of an Unmeasured Confounder},
  author = {Miao, Wang and Geng, Zhi and Tchetgen Tchetgen, Eric J},
  year = {2018},
  month = dec,
  journal = {Biometrika},
  volume = {105},
  number = {4},
  pages = {987--993},
  issn = {0006-3444},
  doi = {10.1093/biomet/asy038},
  urldate = {2024-01-29}
}

@article{mogensenMarkovEquivalenceMarginalized2020,
  title = {Markov Equivalence of Marginalized Local Independence Graphs},
  author = {Mogensen, S{\o}ren Wengel and Hansen, Niels Richard},
  year = {2020},
  month = feb,
  journal = {The Annals of Statistics},
  volume = {48},
  number = {1},
  pages = {539--559},
  publisher = {Institute of Mathematical Statistics},
  issn = {0090-5364, 2168-8966},
  doi = {10.1214/19-AOS1821},
  urldate = {2023-07-04}
}

@book{neapolitanProbabilisticReasoningExpert1989,
  title = {Probabilistic {{Reasoning}} in {{Expert Systems}}: {{Theory}} and {{Algorithms}}},
  shorttitle = {Probabilistic {{Reasoning}} in {{Expert Systems}}},
  author = {Neapolitan, Richard E.},
  year = {1989},
  publisher = {Wiley},
  isbn = {978-0-471-61840-9},
  langid = {english}
}

@article{neyman1923application,
  title = {On the Application of Probability Theory to Agricultural Experiments. Essay on Principles. Section 9.},
  author = {Neyman, Jerzy S},
  year = {1923},
  journal = {Annals of Agricultural Sciences},
  volume = {10},
  pages = {1--51}
}

@inproceedings{pearlBayesianNetwcrksModel1985,
  title = {Bayesian Netwcrks: {{A}} Model Cf Self-Activated Memory for Evidential Reasoning},
  shorttitle = {Bayesian Netwcrks},
  booktitle = {Proceedings of the 7th {{Conference}} of the {{Cognitive Science Society}}},
  author = {Pearl, Judea},
  year = {1985},
  pages = {329--334},
  urldate = {2024-01-29}
}

@article{pearlCausalDiagramsEmpirical1995,
  title = {Causal Diagrams for Empirical Research},
  author = {Pearl, Judea},
  year = {1995},
  month = dec,
  journal = {Biometrika},
  volume = {82},
  number = {4},
  pages = {669--688},
  issn = {0006-3444},
  doi = {10.1093/biomet/82.4.669},
  urldate = {2023-06-22}
}

@book{pearlCausalityModelsReasoning2000,
  title = {Causality:  {{Models}}, Reasoning, and Inference},
  shorttitle = {Causality},
  author = {Pearl, Judea},
  year = {2000},
  series = {Causality:  {{Models}}, Reasoning, and Inference},
  pages = {xvi, 384},
  publisher = {Cambridge University Press},
  address = {New York, NY, US},
  isbn = {978-0-521-77362-1}
}

@article{pearlLinearModelsUseful2013,
  title = {Linear {{Models}}: {{A Useful}} ``{{Microscope}}'' for {{Causal Analysis}}},
  shorttitle = {Linear {{Models}}},
  author = {Pearl, Judea},
  year = {2013},
  month = may,
  journal = {Journal of Causal Inference},
  volume = {1},
  number = {1},
  pages = {155--170},
  publisher = {De Gruyter},
  issn = {2193-3685},
  doi = {10.1515/jci-2013-0003},
  urldate = {2024-01-29},
  copyright = {De Gruyter expressly reserves the right to use all content for commercial text and data mining within the meaning of Section 44b of the German Copyright Act.},
  langid = {english}
}

@book{pearlProbabilisticReasoningIntelligent1988,
  title = {Probabilistic Reasoning in Intelligent Systems: {{Networks}} of Plausible Inference},
  author = {Pearl, Judea},
  year = {1988},
  publisher = {Morgan Kaufmann Publishers Inc.},
  address = {San Francisco, CA, USA},
  isbn = {1-55860-479-0}
}

@article{perkovicCompleteGraphicalCharacterization2018,
  title = {Complete {{Graphical Characterization}} and {{Construction}} of {{Adjustment Sets}} in {{Markov Equivalence Classes}} of {{Ancestral Graphs}}},
  author = {Perkovi{\'c}, Emilija and Textor, Johannes and Kalisch, Markus and Maathuis, Marloes H.},
  year = {2018},
  journal = {Journal of Machine Learning Research},
  volume = {18},
  number = {220},
  pages = {1--62},
  issn = {1533-7928},
  urldate = {2024-01-29}
}

@article{richardson03_markov_proper_acycl_direc_mixed_graph,
  title = {Markov Properties for Acyclic Directed Mixed Graphs},
  author = {Richardson, Thomas},
  year = {2003},
  journal = {Scandinavian Journal of Statistics},
  volume = {30},
  number = {1},
  pages = {145--157},
  doi = {10.1111/1467-9469.00323},
  date_added = {Wed Oct 19 21:18:18 2022}
}

@techreport{richardson2013single,
  title = {Single World Intervention Graphs ({{SWIGs}}): {{A}} Unification of the Counterfactual and Graphical Approaches to Causality},
  author = {Richardson, Thomas S and Robins, James M},
  year = {2013},
  number = {128},
  institution = {{Center for the Statistics and the Social Sciences, University of Washington Series}}
}

@article{richardsonNestedMarkovProperties2023,
  title = {Nested {{Markov}} Properties for Acyclic Directed Mixed Graphs},
  author = {Richardson, Thomas S. and Evans, Robin J. and Robins, James M. and Shpitser, Ilya},
  year = {2023},
  month = feb,
  journal = {The Annals of Statistics},
  volume = {51},
  number = {1},
  pages = {334--361},
  publisher = {Institute of Mathematical Statistics},
  issn = {0090-5364, 2168-8966},
  doi = {10.1214/22-AOS2253},
  urldate = {2023-05-26}
}

@article{rubin1974estimating,
  title = {Estimating Causal Effects of Treatments in Randomized and Nonrandomized Studies.},
  author = {Rubin, Donald B},
  year = {1974},
  journal = {Journal of educational Psychology},
  volume = {66},
  number = {5},
  pages = {688},
  publisher = {American Psychological Association}
}

@article{shpitserAcyclicLinearSEMs2018,
  title = {Acyclic {{Linear SEMs Obey}} the {{Nested Markov Property}}},
  author = {Shpitser, Ilya and Evans, Robin J. and Richardson, Thomas S.},
  year = {2018},
  month = aug,
  journal = {Uncertainty in Artificial Intelligence: Proceedings of the ... Conference. Conference on Uncertainty in Artificial Intelligence},
  volume = {2018},
  pages = {255},
  issn = {1525-3384},
  langid = {english},
  pmcid = {PMC6461354},
  pmid = {30983907}
}

@incollection{shpitserIdentificationGraphicalCausal2018,
  title = {Identification in {{Graphical Causal Models}}},
  booktitle = {Handbook of {{Graphical Models}}},
  author = {Shpitser, Ilya},
  editor = {Maathuis, Marloes and Drton, Mathias and Lauritzen, Steffen and Wainwright, Martin},
  year = {2018},
  publisher = {CRC Press},
  isbn = {978-0-429-46397-6}
}

@inproceedings{shpitserValidityCovariateAdjustment2010,
  title = {On the Validity of Covariate Adjustment for Estimating Causal Effects},
  booktitle = {Proceedings of the {{Twenty-Sixth Conference}} on {{Uncertainty}} in {{Artificial Intelligence}}},
  author = {Shpitser, Ilya and VanderWeele, Tyler and Robins, James M.},
  year = {2010},
  month = jul,
  series = {{{UAI}}'10},
  pages = {527--536},
  publisher = {AUAI Press},
  address = {Arlington, Virginia, USA},
  urldate = {2023-04-18},
  isbn = {978-0-9749039-6-5}
}

@book{spirtesCausationPredictionSearch1993,
  title = {Causation, {{Prediction}}, and {{Search}}},
  author = {Spirtes, Peter and Glymour, Clark and Scheines, Richard},
  editor = {Berger, J. and Fienberg, S. and Gani, J. and Krickeberg, K. and Olkin, I. and Singer, B.},
  year = {1993},
  series = {Lecture {{Notes}} in {{Statistics}}},
  volume = {81},
  publisher = {Springer},
  address = {New York, NY},
  doi = {10.1007/978-1-4612-2748-9},
  urldate = {2023-04-19},
  isbn = {978-1-4612-7650-0 978-1-4612-2748-9}
}

@inproceedings{spirtesDirectedCyclicGraphical1995,
  title = {Directed Cyclic Graphical Representations of Feedback Models},
  booktitle = {Proceedings of the {{Eleventh}} Conference on {{Uncertainty}} in Artificial Intelligence},
  author = {Spirtes, Peter},
  year = {1995},
  month = aug,
  series = {{{UAI}}'95},
  pages = {491--498},
  publisher = {Morgan Kaufmann Publishers Inc.},
  address = {San Francisco, CA, USA},
  urldate = {2024-07-18},
  isbn = {978-1-55860-385-1}
}

@article{sullivantTrekSeparationGaussian2010,
  title = {Trek Separation for {{Gaussian}} Graphical Models},
  author = {Sullivant, Seth and Talaska, Kelli and Draisma, Jan},
  year = {2010},
  month = jun,
  journal = {The Annals of Statistics},
  volume = {38},
  number = {3},
  pages = {1665--1685},
  publisher = {Institute of Mathematical Statistics},
  issn = {0090-5364, 2168-8966},
  doi = {10.1214/09-AOS760},
  urldate = {2023-04-19}
}

@inproceedings{tianTestableImplicationsCausal2002b,
  title = {On the Testable Implications of Causal Models with Hidden Variables},
  booktitle = {Proceedings of the {{Eighteenth}} Conference on {{Uncertainty}} in Artificial Intelligence},
  author = {Tian, Jin and Pearl, Judea},
  year = {2002},
  month = aug,
  series = {{{UAI}}'02},
  pages = {519--527},
  publisher = {Morgan Kaufmann Publishers Inc.},
  address = {San Francisco, CA, USA},
  urldate = {2024-07-18},
  isbn = {978-1-55860-897-9}
}

@inproceedings{vermaEquivalenceSynthesisCausal1990,
  title = {Equivalence and Synthesis of Causal Models},
  booktitle = {Proceedings of the 6th Conference on Uncertainty in Artificial Intelligence ({{UAI-1990}})},
  author = {Verma, Thomas S and Pearl, Judea},
  year = {1990},
  pages = {220--227},
  address = {Cambridge, MA, USA}
}

@article{wright34_method_path_coeff,
  title = {The Method of Path Coefficients},
  author = {Wright, Sewall},
  year = {1934},
  journal = {The Annals of Mathematical Statistics},
  volume = {5},
  number = {3},
  pages = {161--215},
  doi = {10.1214/aoms/1177732676},
  date_added = {Mon Oct 17 11:01:04 2022}
}

@article{wrightRelativeImportanceHeredity1920,
  title = {The Relative Importance of Heredity and Environment in Determining the Piebald Pattern of Guinea-Pigs},
  author = {Wright, Sewall},
  year = {1920},
  journal = {Proceedings of the National Academy of Sciences},
  volume = {6},
  number = {6},
  pages = {320--332},
  doi = {10.1073/pnas.6.6.320},
  date_added = {Mon Jan 25 15:58:52 2021}
}

\clearpage

\appendix

\section{Dictionary and \LaTeX\ macros of the walk algebra}

A dictionary of some special types of walks on a directed mixed graph
can be found \Cref{tab:graph-convention}.

The \texttt{graph-separation} package provides a number of useful
macros and is available for download
at
\url{https://www.statslab.cam.ac.uk/~qz280/files/graph-separation.sty}. The
basic commands provided by \texttt{graph-separation} are
straight and squiglly lines with no, half, or full arrowheads at the
two ends.
\begin{center}
\begin{tabular}{llll}
  $\nostraigno$ & \verb|\nostraigno| & $\nosquigno$ & \verb|\nosquigno| \\
  $\nostraigfull$ & \verb|\nostraigfull| & $\nosquigfull$ & \verb|\nosquigfull| \\
  $\fullstraigno$ & \verb|\fullstraigno| & $\fullsquigno$ & \verb|\fullsquigno| \\
  $\halfstraighalf$ & \verb|\halfstraighalf| & $\halfsquighalf$ & \verb|\halfsquighalf| \\
  $\halfstraigfull$ & \verb|\halfstraigfull| & $\halfsquigfull$ & \verb|\halfsquigfull|  \\
  $\fullstraighalf$ & \verb|\fullstraighalf| & $\fullsquighalf$ & \verb|\fullsquighalf| \\
  $\fullstraigfull$ & \verb|\fullstraigfull| & $\fullsquigfull$ & \verb|\fullsquigfull|
\end{tabular}
\end{center}
From these we can derive other types of walks that may contain an
arbitrary number of colliders. The package provides macros for the
most common types listed below.
\begin{center}
  \begin{tabular}{llll}
  $\udedge$ & \verb|\udedge| &
  $\noudedge$ & \verb|\noudedge| \\
  $\rdedge$ & \verb|\rdedge| &
  $\nordedge$ & \verb|\nordedge| \\
  $\ldedge$ & \verb|\ldedge| &
  $\noldedge$ & \verb|\noldedge| \\
  $\bdedge$ & \verb|\bdedge| &
  $\nobdedge$ & \verb|\nobdedge| \\
  $\rdpath$ & \verb|\rdpath| &
  $\nordpath$ & \verb|\nordpath| \\
  $\ldpath$ & \verb|\ldpath| &
  $\noldpath$ & \verb|\noldpath| \\
  $\trek$ & \verb|\tconnarc| or \verb|\trek| &
  $\notrek$ & \verb|\notconnarc| or \verb|\notrek| \\
  $\mconnarc$ & \verb|\mconnarc| &
  $\nomconnarc$ & \verb|\nomconnarc| \\
  $\dconnarc$ & \verb|\dconnarc| &
  $\nodconnarc$ & \verb|\nodconnarc| \\
  $\confarc$ & \verb|\confarc| &
  $\noconfarc$ & \verb|\noconfarc| \\
  $\samedist$ & \verb|\samedist| &
  $\confpath$ & \verb|\confpath| \\
  $\colliderconn$ & \verb|\colliderconn| &
  $\markovblanket$ & \verb|\markovblanket| \\
  $\uconn$ & \verb|\uconn| & $\tconn$ & \verb|\tconn| \\
  $\mconn$ & \verb|\mconn| &
  $\dconn$ & \verb|\dconn| \\
  $\muconn$ & \verb|\muconn| &
  $\deltaconn$ & \verb|\deltaconn| \\
\end{tabular}
\end{center}

\begin{table}[h] \renewcommand{\arraystretch}{1.3}
  \centering
  \caption{A dictionary of the walk algebra generated by an acyclic
    directed mixed graph.}
  \label{tab:graph-convention}

  \begin{tabular}{lll}
    \toprule
    Notation & Def. & Explanation \\
    \midrule
    $W[V \rdedge V]$ & \eqref{eq:directed-edges} & directed edges \\
    $W[V \bdedge V]$ & \eqref{eq:bidirected-edges} &
                                                     bidirected edges \\
    $W[V \rdpath V]$ & \eqref{eq:directed-walk} & (right-)directed walks \\
    $W[V \ldpath V]$ & \eqref{eq:left-directed-walk} &
                                                       (left-)directed walks \\
    $W[V \trek V]$ & \eqref{eq:trek} & treks \\
    $W[V \dconnarc V]$ & \eqref{eq:dconn-walk} & d-connected walks \\
    $W[V \mconnarc V]$ & \eqref{eq:mconn-walk} & m-connected walks
                                                 (aka arcs) \\
    $P[V \confarc V]$ & \eqref{eq:confarc} & confounding arcs \\
    $W[V \rdpath V \mid L]$ &
                              \eqref{eq:unblocked-directed-walks}
                          & directed walks not blocked by $L$ \\
    $W[V \trek V \mid L]$ &
                            \eqref{eq:unblocked-treks}
                    & treks not blocked by $L$ \\
    $P[V \trek V \mid L]$ & \eqref{eq:trekpath-given-L} & treks not
                                                          blocked by
                                                          $L$ that are
    paths \\
    $P[V \dconnarc V~\text{via root}~V_r]$ & \eqref{eq:dconn-arc-root} &
                                                                       d-connected
                                                                       paths
                                                                       via
                                                                       root
                                                                       $V_r$
    \\
    $W[V \colliderconn V]$ & \eqref{eq:collider-conn} &
                                                        collider-connected
                                                        walks \\
    $W[V \samedist V]$ & \eqref{eq:samedist} &
                                                        a walk
                                               consisting of
                                               bidirected edges \\
    $W[V \mconn V \mid L]$ & \eqref{eq:mconn} & m-connected walks
                                                given $L$ \\
    $W[V \tconn V \mid L]$ &
                             \eqref{eq:tconn}
                          & t-connected walks given $L$ \\
    $P[V \mconn V \mid_a L ]$ &
                                \eqref{eq:mconnpath-given-L}
                          & m-connected paths not ancestrally
                            blocked by $L$ \\

    $P[V \dconn V \mid_a L ]$ &
                                \eqref{eq:dconnpath-given-L}
                          & d-connected paths not ancestrally
                            blocked by $L$ \\
    $P[V \confpath V]$ & \eqref{eq:confpath} & confounding paths \\
    \bottomrule
  \end{tabular}
\end{table}

\end{document}